\documentclass[a4paper]{amsart}

\pdfoutput=1

\usepackage{lmodern}
\usepackage{bbm}
\usepackage{amsmath,amssymb,amsthm,url,scalerel}
\usepackage[utf8]{inputenc}
\usepackage[T1]{fontenc}
\usepackage{libertine}
\usepackage[libertine,cmintegrals,cmbraces]{newtxmath}
\usepackage{verbatim}
\usepackage[mathscr]{euscript}
\usepackage[shortlabels]{enumitem}
\usepackage{stmaryrd}
\usepackage{mathtools}
\usepackage{microtype}
\usepackage{xcolor}
\definecolor{darkred}{RGB}{139,0,0}
\definecolor{darkblue}{RGB}{0,0,139}
\definecolor{darkgreen}{RGB}{0,100,0}
\usepackage{hyperref}
\hypersetup{
  colorlinks   = true,
  urlcolor     = darkred,
  linkcolor    = darkred,
  citecolor   = darkgreen}
\usepackage{tikz}
\usepackage{tikz-cd}
\usepackage[capitalise,noabbrev]{cleveref}
\usepackage{nicefrac}
\usepackage{booktabs}

\AtBeginDocument{%
   \def\MR#1{}
}
\setenumerate[1]{label=(\roman*),}

\DeclareMathAlphabet{\mathpzc}{OT1}{pzc}{m}{it}
\setitemize[1]{label=\raisebox{0.25ex}{\tiny$\bullet$}}

\newcommand{\Diff}{\ensuremath{\mathrm{Diff}}}

\newcommand{\BDiff}{\ensuremath{\mathrm{BDiff}}}

\newcommand{\Homeo}{\ensuremath{\mathrm{Homeo}}}

\newcommand{\BlockBHomeo}{\smash{\ensuremath{\mathrm{B\widetilde{Homeo}}}}}
\newcommand{\BlockHomeo}{\smash{\ensuremath{\mathrm{\widetilde{Homeo}}}}}

\newcommand{\hAut}{\ensuremath{\mathrm{hAut}}}
\newcommand{\BhAut}{\ensuremath{\mathrm{BhAut}}}

\newcommand{\bP}{\mathrm{bP}}
\newcommand{\bSpin}{\mathrm{bSpin}}

\newcommand{\OO}{\mathrm{O}}

\newcommand{\Hom}{\mathrm{Hom}}

\newcommand{\GL}{\mathrm{GL}}

\newcommand{\SL}{\mathrm{SL}}
\newcommand{\PSL}{\mathrm{PSL}}
\newcommand{\BSL}{\mathrm{BSL}}

\newcommand{\coker}{\mathrm{coker}}
\newcommand{\id}{\mathrm{id}}

\newcommand{\im}{\mathrm{im}}
\newcommand{\ext}{\mathrm{ext}}

\newcommand{\inc}{\mathrm{inc}}

\newcommand{\colim}{\mathrm{colim}}

\newcommand{\interior}{\mathrm{int}}
\newcommand{\Tor}{\mathrm{Tor}}

\newcommand{\ev}{\mathrm{ev}}
\newcommand{\Sm}{\mathrm{Sm}}
\newcommand{\con}{\mathrm{con}}

\newcommand{\TOP}{\mathrm{Top}}
\newcommand{\STOP}{\mathrm{STop}}
\newcommand{\BTOP}{\mathrm{BTop}}
\newcommand{\BSTOP}{\mathrm{BSTop}}
\newcommand{\BSO}{\mathrm{BSO}}
\newcommand{\BO}{\mathrm{BO}}

\newcommand{\SO}{\mathrm{SO}}

\newcommand{\Aut}{\mathrm{Aut}}
\newcommand{\Lift}{\mathrm{Lift}}
\newcommand{\pp}{\mathrm{p}}

\newcommand{\diff}{\mathrm{diff}}

\DeclareMathAlphabet{\mathpzc}{OT1}{pzc}{m}{it}
\newcommand{\catsingle}[1]{\ensuremath{\mathcal{#1}}}

\newcommand{\oH}{\ensuremath{\mathrm{H}}}

\newcommand{\oO}{\ensuremath{\mathrm{O}}}

\newcommand{\bfC}{\ensuremath{\mathbf{C}}}

\newcommand{\bfR}{\ensuremath{\mathbf{R}}}

\newcommand{\bfZ}{\ensuremath{\mathbf{Z}}}
\newcommand{\bfQ}{\ensuremath{\mathbf{Q}}}
\newcommand{\bfS}{\ensuremath{\mathbf{S}}}

\newcommand{\bfU}{\ensuremath{\mathbf{U}}}

\newcommand{\cS}{\ensuremath{\catsingle{S}}}
\newcommand{\cT}{\ensuremath{\catsingle{T}}}

\newcommand{\ra}{\rightarrow}

\newcommand{\lra}{\longrightarrow}

\newcommand{\xra}[1]{\xrightarrow{#1}}

\newcommand{\xlra}[1]{\overset{#1}{\longrightarrow}}

\newcommand{\onto}{\twoheadrightarrow}
\newcommand{\hra}{\hookrightarrow}

\newcommand{\what}{\widehat}

\newcommand{\mr}[1]{{\mathrm{#1}}}

\newcommand{\circled}[1]{\raisebox{.5pt}{\textcircled{\raisebox{-.9pt} {#1}}}}
\makeatletter
\renewcommand{\boxed}[1]{\text{\fboxsep=.2em\fbox{\m@th$\displaystyle#1$}}}
\makeatother

    \makeatletter
\newcommand{\mylabel}[2]{#2\def\@currentlabel{#2}\label{#1}}
\makeatother

\newtheorem{bigthm}{Theorem}
\newtheorem{bigcor}[bigthm]{Corollary}

\newtheorem{thm}{Theorem}[section]

\newtheorem{lem}[thm]{Lemma}
\newtheorem{prop}[thm]{Proposition}
\newtheorem{cor}[thm]{Corollary}

\theoremstyle{definition}

\theoremstyle{remark}

\newtheorem{rem}[thm]{Remark}
\newtheorem*{nrem}{Remark}

\calclayout

\begin{document}

\title{Mapping class groups of exotic tori and actions by $\SL_d(\bfZ)$}

\author{Mauricio Bustamante}
\address{Departamento de Matem\'aticas, Pontificia Universidad Cat\'olica de Chile}
\email{bustamante.math@gmail.com}

\author{Manuel Krannich}
\address{Department of Mathematics, Karlsruhe Institute of Technology, 76131 Karlsruhe, Germany}
\email{krannich@kit.edu}

\author{Alexander Kupers}
\address{Department of Computer and Mathematical Sciences, University of Toronto Scarborough, 1265 Military Trail, Toronto, ON M1C 1A4, Canada}
\email{a.kupers@utoronto.ca}

\author{Bena Tshishiku}
\address{Department of Mathematics, Brown University}
\email{bena\_tshishiku@brown.edu}


\begin{abstract}We determine for which exotic tori $\cT$ of dimension $d\neq4$ the homomorphism from the group of isotopy classes of orientation-preserving diffeomorphisms of $\cT$ to $\SL_d(\bfZ)$ given by the action on the first homology group is split surjective. As part of the proof we compute the mapping class group of all exotic tori $\cT$ that are obtained from the standard torus by a connected sum with an exotic sphere. Moreover, we show that any nontrivial $\SL_d(\bfZ)$-action on $\cT$ agrees on homology with the standard action, up to an automorphism of $\SL_d(\bfZ)$. When combined, these results in particular show that many exotic tori do not admit any nontrivial differentiable action by $\SL_d(\bfZ)$. \end{abstract}

\maketitle

\tableofcontents


 A \emph{homotopy $d$-torus} $\cT$ is a $d$-dimensional smooth manifold that is homotopy equivalent to the \emph{standard torus} $T^d=\times^dS^1$ and hence also homeomorphic to it, by a known instance of the Borel conjecture; see \cite{HsiangWallII} for $d>4$, \cite[11.5]{FreedmanQuinn} for $d=4$, \cite[6.5]{Waldhausen3} and the Poincar\'e conjecture for $d=3$. If $\cT$ is not diffeomorphic to the standard torus $T^d$, it is called \emph{exotic}. For instance, given an exotic sphere $\Sigma$ of dimension $k\le d$, the connected sum $(T^k\sharp\Sigma)\times T^{d-k}$ is an exotic $d$-torus.

One of the prominent features of the standard torus $T^d\cong \bfR^d/\bfZ^d$ is that it admits a faithful action $\SL_d(\bfZ)\ra \Diff^+(T^d)$ by $\SL_d(\bfZ)$ through orientation-preserving diffeomorphisms, induced by the linear action of $\SL_d(\bfZ)$ on $\bfR^d$. For a general homotopy $d$-torus $\cT$ one might thus wonder:
\begin{enumerate}[leftmargin=0.8cm]
\item[\mylabel{enum:action}{(A)}]\emph{Is there a faithful action $\SL_d(\bfZ)\ra \Diff^+(\cT)$? If not, is there even any nontrivial action?}\end{enumerate}
As the $\SL_d(\bfZ)$-action on the standard torus splits the homomorphism $\Diff^+(T^d)\ra\SL_d(\bfZ)$ induced by the action on the first homology group $\oH_1(T^d)\cong\pi_1(T^d)\cong \bfZ^d$, it seems natural to approach Question \ref{enum:action} by first considering the following weaker question which is an instance of a high-dimensional version of a Nielsen realisation problem posed by Thurston \cite[Prob.\ 2.6]{kirby}:
\begin{enumerate}[leftmargin=0.8cm]
\item[\mylabel{enum:splitting}{(S)}]\emph{Is the homomorphism $\Diff^+(\cT)\ra \SL_d(\bfZ)$ given by the action on $\oH_1(\cT)$ split surjective?}
\end{enumerate}
This homomorphism factors through the mapping class group $\pi_0\,\Diff^+(\cT)$ of isotopy classes of orientation-preserving diffeomorphisms, so one can weaken the question further to:
\begin{enumerate}[leftmargin=0.8cm]
\item[\mylabel{enum:splitting-isotopy}{(S$_0$)}] \emph{Is the homomorphism $\pi_0\,\Diff^+(\cT)\ra \SL_d(\bfZ)$ given by the action on $\oH_1(\cT)$ split surjective?}
\end{enumerate}
This work establishes several results regarding these three questions. Note that a positive answer to \ref{enum:splitting} implies positive answers to \ref{enum:action} and \ref{enum:splitting-isotopy}. As part of our results, we
\begin{itemize}[leftmargin=0.8cm]
\item answer Question \ref{enum:splitting-isotopy} in all dimensions $d\neq4$,
\item show that Questions \ref{enum:splitting} and \ref{enum:action} are in fact equivalent, and
\item conclude that for many exotic tori the answer to all three questions is negative.
\end{itemize}
In what follows, we describe these results and various extensions of them in more detail.

\subsection*{Splitting the homology action up to isotopy}
Our first main result answers \ref{enum:splitting-isotopy} for $d\neq4$:

\begin{bigthm}\label{bigthm:splitting}
For a homotopy torus $\cT$ of dimension $d\neq4$, the morphism
\[\pi_0\,\Diff^+(\cT)\lra \SL_d(\bfZ)\]
induced by the action on $\oH_1(\cT)$ is split surjective if and only if $\cT$ is diffeomorphic to $T^d\sharp\Sigma$ for a homotopy sphere $\Sigma\in \Theta_d$ such that $\eta\cdot \Sigma\in\Theta_{d+1}$ is divisible by $2$ in the abelian group $\Theta_{d+1}$.
\end{bigthm}

Here $\Theta_d$ is Kervaire--Milnor's finite abelian group of homotopy $d$-spheres \cite{KervaireMilnor} and $\eta\cdot \Sigma\in \Theta_{d+1}$ for $\Sigma\in \Theta_d$ is the value of $\eta \otimes \Sigma$ under the Milnor--Munkres--Novikov pairing $\pi_1\,\bfS \otimes \Theta_d \to \Theta_{d+1}$ where $\eta \in \pi_1\,\bfS \cong \bfZ/2$ is the generator of the first stable homotopy group of spheres (see \cite{Bredon} for more on this pairing). The question whether $\eta\cdot \Sigma\in\Theta_{d+1}$ for a given $\Sigma\in\Theta_d$ is divisible by $2$ can in most instances be reduced to a problem in stable homotopy theory which can in turn be solved in many cases. This approach is discussed in \cref{section:ct-homotopy-spheres}, but to already illustrate its practicability at this point, we display in \cref{tab:split} below the first groups of homotopy spheres $\Theta_d$ together with the subgroups $\smash{\Theta^{\mathrm{split}}_d\le \Theta_d}$ of \emph{split spheres}, i.e.\,those $\Sigma\in\Theta_{d}$ for which $\eta\cdot \Sigma$ is divisible by $2$, which is by \cref{bigthm:splitting} equivalent to $\pi_0\,\Diff^+(T^d\sharp\Sigma)\ra \SL_d(\bfZ)$ being split. Note that among the dimensions $d$ for which $\Theta_d$ is nontrivial, there are dimensions in which all spheres are split such as $d=7$, dimensions in which none are split such as $d=8$, as well as dimensions in which some but not all are split such as $d=9$. In \cref{section:ct-homotopy-spheres} we also explain why both cases---the sphere $\Sigma$ being split or not---occur for exotic spheres $\Sigma$ in infinitely many dimensions.

\vspace{+0.7cm}

\begin{table}[h!]
\begin{tabular}{c|c|c|c|c|c|c|c|c}
\toprule
 $d$ & $\le 6$ and  $\neq 4$ & $7$ & $8$ &$9$ &$10$ &$11$ &$12$ &$13$  \\ \midrule
 $\Theta_d$ &$0$ &$\bfZ/28$& $\bfZ/2$ & $\ \ \ \ (\bfZ/2)^{\oplus 2}\oplus \bfZ/2$ &$\bfZ/6$&$\bfZ/992$&$0$&$\bfZ/3$\\
 $\ \ \,\,\Theta^{\mathrm{split}}_d$ &$0$ &$\bfZ/28$& $0$ & $(\bfZ/2)^{\oplus 2}\oplus0$ &$\bfZ/6$&$\bfZ/992$&$0$&$\bfZ/3$\\ \bottomrule
\end{tabular}

\medskip

\begin{tabular}{c|c|c|c|c|c}
\toprule
$14$ & $15$& $16$ &$17$ &$18$ &$19$ \\ \midrule
$\bfZ/2$& $\bfZ/2\oplus\bfZ/8128$&$\bfZ/2$&$\ \ \ \,\,(\bfZ/2)^{\oplus 3}\oplus\bfZ/2 $&$\bfZ/8\oplus \bfZ/2$&\ \ \ \ $\bfZ/2\oplus\bfZ/523264$\ \ \\
$0$ & $\bfZ/2\oplus\bfZ/8128$&$0$&$(\bfZ/2)^{\oplus 3}\oplus0$&$\bfZ/8\oplus \bfZ/2$&\ \ \ \ $\bfZ/2\oplus\bfZ/523264$\ \ \\ \bottomrule
\end{tabular}
\vspace{.5cm}
\caption{The groups $\Theta_d$ of homotopy $d$-spheres for $d\le 19$ together with the subgroups $\smash{\Theta^{\mathrm{split}}_d\le \Theta_d}$ of those $\Sigma\in\Theta_{d}$ for which $\eta\cdot \Sigma$ is divisible by $2$.}
\label{tab:split}
\end{table}

\subsection*{Actions of $\SL_d(\bfZ)$ on homotopy tori}
Our second main result shows that all nontrivial $\SL_d(\bfZ)$-actions on homotopy tori agree on homology with the standard action up to an automorphism.

\begin{bigthm}\label{thm:action-implies-splitting}
Fix $d\ge 3$, a homotopy $d$-torus $\cT$, and an automorphism group \[G\in \{\Diff^+(\cT),\Homeo^+(\cT)\}.\] Any homomorphism $\SL_d(\bfZ)\ra G$ is either trivial or has the property that its postcomposition
\[\SL_d(\bfZ)\lra G\lra\SL_d(\bfZ)\]
with the action on $\oH_1(\cT)$ is an automorphism. Moreover, if also $d\neq 4,5$, then the same holds when replacing $G$ by the group $\pi_0\,G$ of isotopy classes.\end{bigthm}

In particular, given any nontrivial homomorphism $\varphi\colon \SL_d(\bfZ)\ra G$, we obtain a splitting of the action $\alpha\colon G\ra \SL_d(\bfZ)$ on first homology, given by $\varphi\circ (\alpha\circ \varphi)^{-1}$. Applying this to $G=\Diff^+(\cT)$ shows that the above questions  \ref{enum:splitting} and \ref{enum:action} are in fact equivalent. Applying it to $\pi_0\,G=\pi_0\,\Diff^+(\cT)$ also shows that \ref{enum:splitting-isotopy} is equivalent to the following isotopy-analogue of \ref{enum:action}.

\begin{enumerate}[leftmargin=0.8cm]
\item[\mylabel{enum:action-isotopy}{(A$_0$)}]\emph{Is there a faithful action $\SL_d(\bfZ)\ra \pi_0\,\Diff^+(\cT)$? If not, is there even any nontrivial action?}\end{enumerate}
Combining these implications with \cref{bigthm:splitting} results in the following corollary which answers all questions \ref{enum:action}, \ref{enum:splitting},  \ref{enum:splitting-isotopy}, \ref{enum:action-isotopy} in the negative for a large class of homotopy tori and partially answers Question 1.4 and Problem 1.5 in work of Bustamante and Tshishiku \cite{BustamanteTshishiku}.
\begin{bigcor}\label{cor:no-action}
Let $\cT$ be a homotopy torus of dimension $d\neq4$. If
\begin{enumerate}
\item  $\cT$ is not diffeomorphic to a connected sum $T^d\sharp\Sigma$ with $\Sigma\in \Theta_d$, or
\item $\cT$ is diffeomorphic to $T^d\sharp\Sigma$ for some $\Sigma\in\Theta_d$ such that $\eta\cdot \Sigma\in\Theta_{d+1}$ is not divisible by $2$,
\end{enumerate}
then every homomorphism from $\SL_d(\bfZ)$ to $\Diff^+(\cT)$ or to $\pi_0\,\Diff^+(\cT)$ is trivial.
\end{bigcor}

\begin{nrem}[The Zimmer programme]One motivation for considering Question \ref{enum:action} stems from the \emph{Zimmer programme}, part of which studies actions of $\SL_d(\bfZ)$ on manifolds. For instance, it follows from a version of Zimmer's conjecture, now a theorem due to Brown--Fisher--Hurtado \cite{BFH}, that $\SL_d(\bfZ)$ does not act faithfully on smooth manifolds of dimension $\le d-2$. For actions of $\SL_d(\bfZ)$ on $d$-manifolds, there is a conjectural classification by Fisher--Melnick \cite[Conjecture 3.6]{fisher-melnick} which would imply that if $\SL_d(\bfZ)$ acts faithfully on a homotopy $d$-torus $\cT$, then $\cT$ is the standard torus. \cref{cor:no-action} implies this for a large class of homotopy tori.
\end{nrem}

\begin{nrem}[Regularity] Our results are phrased in terms of the group $\Diff^+(\cT)$ of $C^\infty$-diffeomorphisms, but they also hold for the groups $\Diff^{+,k}(\cT)$ of $C^k$-diffeomorphisms for finite $ k\ge1$. For Theorems \ref{bigthm:splitting} and \ref{bigthm:mcg-homotopy-tori}, this follows from the isomorphism $\pi_0\,\Diff^{+}(\cT)\cong \pi_0\,\Diff^{+,k}(\cT)$. For \cref{thm:action-implies-splitting} it follows from the observation that the statement for the group $\Homeo^+(\cT)$ also implies the statement for all its subgroups. The deduction of  \cref{cor:no-action} from Theorems \ref{bigthm:splitting} and \ref{thm:action-implies-splitting} works the same way. In particular, this shows that homotopy tori as in \cref{cor:no-action} do not admit any $C^1$-action by $\SL_d(\bfZ)$.
\end{nrem}

We conclude this introduction by explaining two results featuring in the proofs of  \cref{bigthm:splitting} and \cref{thm:action-implies-splitting} that may be of independent interest.

\subsection*{Mapping class groups of exotic tori}
As an ingredient for the proof of \cref{bigthm:splitting}, we determine the mapping class groups $\pi_0\,\Diff^+(\cT)$ of exotic tori of the form $\cT=T^d\sharp\Sigma$ for $\Sigma\in \Theta_d$ in all dimensions $d\ge7$ in terms of the known mapping class group $\pi_0\,\Diff^+(T^d)$ of the standard torus. Note that $\Theta_d$ is trivial when $d\le 6$ and $d\neq 4$, so in these cases there is nothing to show. To state the result, we first recall the previously known description of $\pi_0\,\Diff^+(T^d)$. As mentioned above, the action of $\SL_d(\bfZ)$ on $T^d$ induces a splitting of the action map $\pi_0\,\Diff^+(T^d)\ra\SL_d(\bfZ)$, so there is a semidirect product decomposition
\[\pi_0\,\Diff^+(T^d)= \SL_d(\bfZ)\ltimes \pi_0\,\Tor^\Diff(T^d)\quad\text{with}\quad \pi_0\,\Tor^\Diff(T^d)\coloneq\ker\big(\pi_0\,\Diff^+(T^d)\ra\SL_d(\bfZ)\big).\]
For $d\ge6$, the kernel $\pi_0\,\Tor^\Diff(T^d)$ is abelian and isomorphic to the sum of $\bfZ[\SL_d(\bfZ)]$-modules \begin{equation}\label{equ:torelli}\textstyle{
 \Omega\coloneq \Big(\bigoplus_{0 \leq j \leq d} (\Lambda^{j} \bfZ^d)\otimes \Theta_{d-j+1}\Big)\oplus \Big( (\Lambda^{d-2} \bfZ^d)\otimes \bfZ/2\Big)\oplus   \Big((\bfZ/2)[\bfZ^d]/(\bfZ/2)[1]\Big)_{C_2}}\end{equation}
where $\SL_d(\bfZ)$ acts through the standard action on $\bfZ^d$, and $(-)_{C_2}$ denotes the coinvariants with respect to the involution induced by multiplication by $-1$ on $\bfZ^d$ (see \cite[Theorem 4.1, Remark (3) on p.~9]{HatcherConcordance}\footnote{\cite[Theorem 4.1]{HatcherConcordance} asserts that the computation of $ \pi_0\,\Tor^\Diff(T^d)$ also holds for $d=5$. However, this relies on a claim attributed to Igusa (see the middle of p.\,7 loc.cit.) for which---to our knowledge---no proof has been provided so far.} and \cite[Theorem 2.5]{HsiangSharpe}). In addition to
this description of $\pi_0\,\Tor^\Diff(T^d)$, our identification of $\pi_0\,\Diff^+(T^d\sharp\Sigma)$ involves the aforementioned homotopy sphere $\eta\cdot\Sigma\in\Theta_{d+1}$ and the unique nontrivial central extension
\[0\lra\bfZ/2\lra\overline{\SL}_d(\bfZ)\lra \SL_d(\bfZ)\lra 0\]
of $\SL_d(\bfZ)$ by $\bfZ/2$; see \cref{sec:extensions}. Our result identifies the group $\pi_0\,\Diff^+(T^d\#\Sigma)$ as a semidirect product of $\SL_d(\bfZ)$ or $\overline{\SL}_d(\bfZ)$ acting on a quotient of $\Omega$ by a nontrivial subgroup depending on $\Sigma$ which is contained in the summand $\Theta_{d+1}\oplus (\bfZ^d\otimes\Theta_d)$ of \eqref{equ:torelli} corresponding to the terms $j=0,1$.

\begin{bigthm}\label{bigthm:mcg-homotopy-tori} For a homotopy sphere $\Sigma\in\Theta_{d}$ of dimension $d\ge7$, there is an isomorphism
\[\pi_0\,\Diff^+(T^d\#\Sigma)\cong\begin{cases}\overline{\SL}_d(\bfZ)\ltimes\Big[
 \Omega/\big(\langle\eta\cdot\Sigma\rangle \oplus (\bfZ^d\otimes \langle\Sigma\rangle\big) \Big]&\text{if }\eta\cdot \Sigma\in\Theta_{d+1}\text{is not divisible by }2\\
\hfil \SL_d(\bfZ)\ltimes\Big[\Omega/ \big(\bfZ^d\otimes \langle\Sigma\rangle\big)\Big] &\text{if }\eta\cdot \Sigma\in\Theta_{d+1}\text{ is divisible by }2
\end{cases}\]
which is compatible with the homomorphisms to $\SL_d(\bfZ)$.
\end{bigthm}

In particular, this result shows that the mapping class group $\pi_0\,\Diff^+(T^d\#\Sigma)$ for $\Sigma\in\Theta_d$ is given by a quotient of $\overline{\SL}_d(\bfZ)\ltimes\Omega$ by a finite abelian subgroup which is always of order at least $2$ and has order precisely $2$ if and only if $\Sigma$ is the standard sphere, so from this mapping class point of view the standard torus admits ``the most symmetries'', as one would expect.

\subsection*{Endomorphisms of $\SL_d(\bfZ)$}
As an ingredient for the proof of \cref{thm:action-implies-splitting}, we prove the following classification results for endomorphisms of $\SL_d(\bfZ)$ for $d\ge3$:

\begin{bigthm}\label{thm:SLdZ}
Fix $d\ge3$. Every nontrivial endomorphism of $\SL_d(\bfZ)$ is an automorphism. Moreover, all automorphisms of $\SL_d(\bfZ)$ agree, up to postcomposition with a conjugation by an element in $\GL_d(\bfZ)$, with either the identity or the inverse-transpose automorphism.
\end{bigthm}

\begin{nrem}Some comments on \cref{thm:SLdZ}.
\begin{enumerate}[leftmargin=*]
\item The proof is ``elementary'' in that it does neither rely on Margulis' superrigidity or normal subgroup theorem, nor on the congruence subgroup property. Using these results, there are likely other proofs. The argument we give was hinted at by Ian Agol in a comment to a question on MathOverflow \cite{mathoverflow} and sketched by Uri Bader in the case $d=3$ as a response to the question (however this sketch has a small gap; see Remarks \ref{rmk:bader} and \ref{rmk:counterexample}).
\item For $d=2$, the statement of \cref{thm:SLdZ} fails: consider the composition
\[\SL_2(\bfZ)\lra \oH_1(\SL_2(\bfZ))\cong\bfZ/12\lra\SL_2(\bfZ)\]
where the first arrow is abelianisation and the second sends a generator to $-\id\in\SL_2(\bfZ)$.
\item The second part of \cref{thm:SLdZ} holds more generally; see \cite[Theorem A]{OMeara}.
\end{enumerate}
\end{nrem}

\subsection*{Acknowledgements}
We would like to thank Wilberd van der Kallen for helpful comments. AK acknowledges the support of the Natural Sciences and Engineering Research Council of Canada (NSERC) [funding reference number 512156 and 512250]. AK was supported by an Alfred J.~Sloan Research Fellowship. BT is supported by NSF grant DMS-2104346. MB is supported by ANID Fondecyt Iniciaci\'on en Investigaci\'on grant 11220330.

\section{Collar twists} As preparation to the proof of Theorems \ref{bigthm:splitting} and \ref{bigthm:mcg-homotopy-tori}, we collect various results on a certain map $\SO(d)\ra \BDiff_\partial(M\backslash{\interior(D^d)})$ defined by twisting a collar of the complement of an embedded $d$-disc in a closed smooth $d$-manifold $M$. After explaining the construction, we discuss how this map behaves under taking products and connected sums, followed by some results on the collar twisting map for specific choices of $M$, first homotopy spheres and then homotopy tori.

\subsection{The collar twist}
Given a closed connected oriented $d$-dimensional manifold $M$, we write
\[M^\circ \coloneq M \backslash \interior(D^d)\]
for the complement of a fixed embedded disc $D^d \subset M$ that is compatible with the orientation (which is unique up to isotopy), and we write $\Diff_\partial(M^\circ)$ for the group of diffeomorphisms of $M^\circ$ that fix a neighbourhood of the boundary sphere $\partial M^\circ = S^{d-1}$ pointwise, equipped with the smooth topology. The latter is homotopy equivalent to the larger group $\Diff_{T_*M}(M)$ of diffeomorphisms of $M$ that fix the centre of the disc $\ast\in M$ as well as the tangent space at this point. The group $\Diff_{T_*M}(M)$ is the fibre of the fibration $\smash{d\colon \Diff_{\ast}^+(M)\ra \GL^{+}_d(\bfR)}$ assigning to a diffeomorphism that fixes $\ast$ its (orientation-preserving) derivative at that point, so after delooping and using the equivalence $\smash{\GL^{+}_d(\bfR) \simeq \SO(d)}$, there is a homotopy fibration sequence
\begin{equation}\label{equ:collar-twist-sequence}
\BDiff_\partial(M^\circ)\xlra{\ext} \BDiff^+_\ast(M)\xlra{d} \BSO(d).
\end{equation}
where $\ext$ is induced by extending a diffeomorphism of $M^{\circ}$ to $M$ by the identity. The connecting map $\SO(d)\simeq\Omega\BSO(d)\ra \BDiff_\partial(M^\circ)$ has the following geometric description: there is a homomorphism $\Omega\SO(d)\ra \Diff_\partial([0,1] \times S^{d-1})$ which sends a smooth loop $\gamma\in\Omega\SO(d)$ that is constant near the endpoints to the self-diffeomorphism of $[0,1] \times S^{d-1}$ given by mapping  $(t,x)$ to $(t,\gamma(t)\cdot x)$, and a homomorphism $\ext\colon \Diff_\partial([0,1] \times S^{d-1})\ra \Diff_\partial(M^{\circ})$ induced by a choice of collar of the boundary sphere in $M^{\circ}$. Delooping their composition gives a map
\[
\Upsilon_M\colon \SO(d)\lra \BDiff_\partial(M^\circ)
\]
that agrees with the aforementioned connecting map; see e.g.\,\cite[p.\,9]{KrannichExotic}. Following Section 3 of loc.cit., we call $\Upsilon_M$ the \emph{collar twisting map of $M$}. This map is relevant to the study of the mapping class groups of $M$ and $M^\circ$, since the sequence \eqref{equ:collar-twist-sequence} induces an exact sequence of groups
\begin{equation}\label{equ:ct-exact-sequence}\left(\pi_1\,\SO(d)\cong\begin{cases}\bfZ&\text{ if }d=2\\
\bfZ/2&\text{ if }d\ge3\\
\end{cases}\right) \xlra{(\Upsilon_M)_\ast} \pi_0\,\Diff_\partial(M^\circ)\xlra{\ext} \pi_0\,\Diff^+_\ast(M)\lra 0,\end{equation} so the second morphism in this sequence is an isomorphism if and only if the image
\[t_M \coloneqq (\Upsilon_M)_\ast(1) \in\pi_0\,\Diff_\partial(M^{\circ})\]
of the standard generator of the leftmost group under the first map $(\Upsilon_M)_\ast$ is trivial. We call this element the \emph{collar twist of $M$}. Note that the collar twist lies in the centre of $\pi_0\,\Diff_\partial(M^{\circ})$, because the image of the connecting map $\pi_2(X)\ra \pi_1(F)$ in the long exact sequence of homotopy groups for any fibration $F\ra E\ra X$ has this property. Alternatively, one could use that the collar twist is supported in a collar and that every diffeomorphism fixing boundary can be isotoped to also fix any chosen collar, thereby having disjoint support from the collar twist.

\subsection{Collar twists of products and connected sums}
The following proposition shows that collar twisting maps behave well with respect to products and connected sums. Here and in what follows, we identify $(M\sharp N)^\circ$ with the boundary connected sum  $M^\circ\natural N^\circ$ via the preferred isotopy class of diffeomorphisms between these two manifolds.

\begin{prop}\label{prop:ct-ctdsum-product}Let $M$ and $N$ be closed oriented connected manifolds of dimension $m$ and $n$.
\begin{enumerate}
\item \label{enum:ct-product} The compositions
\[\SO(m)\subset\SO(m+n)\xlra{\Upsilon_{M\times N}} \BDiff_\partial((M \times N)^\circ)\]
and
\[\SO(m) \xlra{\Upsilon_M}\BDiff_\partial(M^\circ)\xrightarrow{(-) \times \id_N}\BDiff_\partial(M^\circ \times N)\xlra{\ext}\BDiff_\partial((M \times N)^\circ)\]
are homotopic. In particular,
		\[t_{M\times N}=t_M\times\id_N\in \pi_0\,\Diff_\partial((M \times N)^\circ)\quad\text{for }m\ge2.
		\]
\item \label{enum:ct-ctdsum} If $M$ and $N$ are of the same dimension $d=m=n$, then the map
\[\SO(d)\xlra{\Upsilon_{M\sharp N}} \BDiff_\partial((M \sharp N)^\circ)=\BDiff_\partial(M^\circ\natural N^\circ)\]
and the composition
\[\hspace{1.2cm}\SO(d)\xlra{\mr{diag}}\SO(d)\times\SO(d)\xrightarrow{\Upsilon_{M}\times \Upsilon_{N}}\BDiff_\partial(M^\circ)\times \BDiff_\partial(N^\circ)\xrightarrow{(-)\natural(-)}\BDiff_\partial(M^\circ\natural N^\circ)\]
are homotopic after restriction to the subspace $\SO(d-1)\subset\SO(d)$. In particular, we have
		\[t_{M\sharp N}=(t_{M}\natural\id_{N^\circ})+(\id_{M^\circ}\natural t_{N})\in \pi_0\,\Diff_\partial(M^\circ \natural N^\circ)\quad\text{for }d\ge3.
		\]
\end{enumerate}
\end{prop}

In order to prove \cref{prop:ct-ctdsum-product}, it is convenient to view the collar twisting map as the instance $P=\ast$ of a more general construction for a compact smooth $p$-dimensional manifold $P$ equipped with an embedding $P\times D^{d-p}\subset M$.
First, one extends the latter inclusion to an embedding $\smash{\overline{P}\times D_2^{d-p}\subset M}$ where $\smash{D_2^{d-p}\subset \bfR^{d-p}}$ is the disc of radius $2$ and $\overline{P} \coloneqq P \cup_{\partial P\times\{0\}} (\partial P \times [0,1])$ is obtained by attaching an external collar to $P$. This extension is unique up to isotopy. Given a smooth function $\lambda \colon \smash{\overline{P}} \times [0,2] \to [0,1]$ and a smooth loop $\gamma \in \Omega \SO(d-p)$ that is constant near the endpoints, consider the diffeomorphism $\smash{\phi_\lambda(\gamma) \colon \overline{P} \times D_2^{d-p} \ra \overline{P} \times D_2^{d-p}}$ by sending $(p,x)$ to $(p,\gamma(\lambda(p,\| x\|))\cdot x)$. In other words, thinking of $\smash{\overline{P} \times D_2^{d-p}}$ as foliated by the leaves $\smash{S_{p,r}\coloneq \{p\} \times D_r^{d-p}}$ for $\smash{p \in \overline{P}}$ and $r \in [0,2]$, the diffeomorphism $\phi_\lambda(\gamma)$ preserves the leaves and acts on the leaf $S_{p,r}$ by rotation with the element at time $\lambda(p,r)$ of the loop $\gamma$. If one additionally assumes that
\begin{enumerate}
	\item\label{enum:lambda-i} $\lambda= 1$ on a neighbourhood of $P \times D^{d-p}$ where $\smash{D^{d-p}=D^{d-p}_1\subset D_2^{d-p}}$ is the unit disc,
	\item\label{enum:lambda-ii} $\lambda = 0$ on a neighbourhood of $\partial(\overline{P} \times D_2^{d-p})$,
\end{enumerate}
then $\phi_\lambda(\gamma)$ agrees with the identity on a neighbourhood of $P \times D^{d-p}\subset \overline{P} \times D_2^{d-p}$ so restricts to a diffeomorphism of the complement. This diffeomorphism of the complement extends via the identity to a  diffeomorphism of $M \backslash \mr{int}(P \times D^{d-p})$ fixing a neighbourhood of the boundary pointwise, so we obtain a map
\[ \phi_\lambda(-)\colon \Omega \SO(d-p) \lra \Diff_\partial(M \backslash \mr{int}(P \times D^{d-p})) \]
which depends continuously on $\lambda$ and is a homomorphism with respect to pointwise multiplication on the domain and composition on the target. Since the space of smooth functions $\lambda$ satisfying (i) and (ii) is contractible by linear interpolation, the delooping of $\phi_{\lambda}(-)$
\[\Phi_P \colon \SO(d-p) \lra \BDiff_\partial(M \backslash \mr{int}(P \times D^{d-p}))\] is independent of $\lambda$ up to homotopy, so only depends on the isotopy class of the embedding $P\times D^{d-p}\subset M$. This map generalises the collar twisting map in the following sense.

\begin{lem}\label{lem:twist-is-t}For $0 \leq p \leq d$, the maps \vspace{-0.1cm}
\[
\def\arraystretch{1.5}
\begin{array}{c@{\hskip 0.05cm} l@{\hskip 0.05cm} c@{\hskip 0.1cm} l@{\hskip 0.1cm} l@{\hskip 0cm} l@{\hskip 0cm}}
\Upsilon_M|_{\SO(d-p)}&\colon&  \SO(d-p)&\lra&\BDiff_\partial(M^\circ)\quad\text{and}\\
\Phi_{D^p}&\colon&\SO(d-p)&\lra& \BDiff_\partial(M\backslash \mr{int}(D^p \times D^{d-p}))\overset{\ext}{\simeq} \BDiff_\partial(M^\circ).\\
\end{array}
\]
are homotopic. Here the embedding $D^p\times D^{d-p}\subset M$ is chosen to be compatible with the orientation.
\end{lem}

\begin{proof}It suffices to show that the two maps $\Omega\SO(d-p)\ra \Diff_\partial(M^\circ)$ before delooping are homotopic as maps of topological groups. Going through the construction, one sees that both maps are instances of the following construction applied to smooth loops $\gamma\in\SO(d-p)$ that are constant near the ends: pick a smooth map $\smash{\lambda \colon D^p_2 \times [0,2] \to [0,1]}$ which is $0$ in a neighbourhood of $\smash{\partial(D^p_2 \times D_2^{d-p})}$, and $1$ in a neighbourhood of $D^d$, consider the self-diffeomorphism of $\smash{D^{p-2}_2\times D_2^p}$ sending $(p,x)$ to  $(p,\gamma(\lambda(p,\|x\|))\cdot x)$, restrict it to a diffeomorphism of $\smash{(D^{p-2}_2\times D^p)\backslash \interior(D^d)}$, and extend the result to a diffeomorphism of $M^{\circ}$ by the identity. As the space of choices for $\lambda$ is contractible by linear interpolation, all maps constructed this way are homotopic.
\end{proof}

A similar argument also shows the following naturality property of the map $\Phi_P$.

\begin{lem}\label{lem:extend-t} Given a compact submanifold $Q \subset \interior(P)$ of codimension $0$, the map
\[\SO(d-p) \xlra{\Phi_Q} \BDiff_\partial(M \backslash \mr{int}(Q \times D^{d-p}))\]
and the composition
\[\SO(d-p) \overset{\Phi_P}\lra \BDiff_\partial(M \backslash \interior(P \times D^{d-p})) \overset{\ext}\lra \BDiff_\partial(M \backslash \interior(Q \times D^{d-p}))\]
are homotopic. Here the embedding $Q\times D^{d-p}\subset M$ is the restriction of the embedding $P\times D^{d-p}\subset M$.\end{lem}

Equipped with Lemmas \ref{lem:twist-is-t} and \ref{lem:extend-t}, we now turn to the proof of \cref{prop:ct-ctdsum-product}.

\begin{proof}[Proof of \cref{prop:ct-ctdsum-product}] For part \ref{enum:ct-product}, note that the composition $\SO(m) \to \BDiff_\partial(M^\circ \times N)$ is an instance of $\Phi_N$ using the embedding $D^m\times N\subset M\times N$, so its postcomposition with $\ext\colon \BDiff_\partial(M^\circ\times N)\ra  \BDiff_\partial((M\times N)\backslash \interior(D^m\times D^n))$ is homotopic to $\Phi_{D^n}$ by \cref{lem:extend-t}, which in turn implies the claim as a result of \cref{lem:twist-is-t}. For part \ref{enum:ct-ctdsum}, view $(M\sharp N)^\circ$ as being obtained from $M^\circ \sqcup N^\circ$ by gluing on a pair-of-pants bordism $W \colon S^{d-1} \sqcup S^{d-1} \leadsto S^{d-1}$. To show the claim, it suffices to show that the maps $t_\mr{in},t_\mr{out} \colon \SO(d-1) \ra \BDiff_\partial(W)$ are homotopic, where $t_\mr{in}$ simultaneously twists collars of the two incoming boundary spheres and  $t_\mr{out}$ twists a collar of the outgoing boundary sphere. Viewing $W$ as $D^d\backslash \interior((e(D^1\sqcup D^1))\times D^{d-1})$ for an embedding $e\colon D^1\sqcup D^1\hookrightarrow\interior(D^1)$, the map $t_\mr{in}$ is given by $\Upsilon_{D^1\sqcup D^1}\colon \SO(d-1)\ra \BDiff_\partial(D^d\backslash \interior(e(D^1\sqcup D^1)\times D^{d-1}))$ and the map $t_{\mr{out}}$ as the composition of $\Upsilon_{D^1}\colon \SO(d-1)\ra \BDiff_\partial(D^d\backslash \interior(D^1\times D^{d-1})$ with $\ext\colon \BDiff_\partial(D^d\backslash \interior(D^1\times D^{d-1}))\ra \BDiff_\partial(D^d\backslash \interior((e(D^1\sqcup D^1))\times D^{d-1}))$, so the claim follows from \cref{lem:extend-t} applied to $P=D^1$, $Q=D^1\sqcup D^1$, and $M=D^d$.
\end{proof}

\begin{rem}
\cref{prop:ct-ctdsum-product} \ref{enum:ct-ctdsum} is a more general version of the ``pants relation" in \cite[Lemma 2.5]{chen-tshishiku}, where a proof of this relation is given by constructing an explicit isotopy.
\end{rem}

\subsection{Collar twists of exotic spheres}\label{section:ct-homotopy-spheres}We now turn to the collar twisting map $\Upsilon_{\Sigma}$ for homotopy spheres $\Sigma$, but we actually restrict our attention to the collar twist $t_{\Sigma}\in\pi_0\,\Diff_\partial(\Sigma^\circ)$ it induces on fundamental groups. We begin with a recollection of the classification of homotopy spheres.

\subsubsection{Classification of homotopy spheres}\label{sec:classification}Recall (e.g.\,from \cite[p.\,90-91]{LevineHomotopySpheres}) that Kervaire--Milnor's finite abelian group $\Theta_d$ of homotopy $d$-spheres \cite{KervaireMilnor} fits for $d \geq 5$ into an exact sequence \begin{equation}\label{equ:KM-sequ}0\ra \bP_{d+1}\ra \Theta_{d}\xra{[-]} \coker(J)_{d}\ra\begin{cases}\bfZ/2&\text{if }d= 2^k-2,\text{ for some }k\\0&\text{otherwise}
\end{cases}\end{equation}
where $\bP_{d+1}\le \Theta_{d}$ is a certain cyclic subgroup and $\coker(J)_d$ is the cokernel of the stable $J$-homomorphism $\pi_d\,\oO\ra\pi_d\,\bfS$ from the homotopy groups of the stable orthogonal group (which are known by Bott periodicity) to the stable homotopy groups of spheres. The order of the cyclic subgroup $\bP_{d+1}$ is known in all cases except $d=125$ (combine \cite[Corollaries 2.2, 3.20, Theorem 4.9]{LevineHomotopySpheres} with \cite[Theorem 1.3]{HHR}):
\[\sharp\bP_{d+1}=\begin{cases}
2^{2k-2}(2^{2k-1}-1)\,\mathrm{num}(4|B_{2k}|/k)&\text{if }d=4k-1\text{ for $k\ge2$}\\
2&\text{if }d=4k+1\text{ for $k\ge1$}\text{ but }d\neq 2^{k}-3\text{ if }k\le 7\\
0&\text{if }d\text{ is even or if }d= 2^{k}-3\text{ for }k\le 6\\
2\text{ or }0&\text{if }d=2^{k}-3\text{ for }k=7.
\end{cases}\]
The map $\coker(J)_{d}\ra\bfZ/2$ in the sequence \eqref{equ:KM-sequ} is known to be trivial as long as $d\neq2^{k}-2$ for $k> 7$. It is known to be nontrivial for $k\le6$, but the case $k=7$ (i.e. $d=126$) is still open (see \cite[Theorem 1.4]{HHR}). The question whether $\bP_{126}=0$ or $\bP_{126}=\bfZ/2$ and the question whether $\coker(J)_{126}\ra \bfZ/2$ is surjective or not (these questions turn out to be equivalent; see \cite[p.\,88]{Levine}) is the last remaining case of the \emph{Kervaire invariant one problem}. The upshot of this discussion is that apart from the two problematic dimensions $d=125, 126$, the group $\Theta_d$ is described in terms of the group $\coker(J)_{d}$ up to extension problems. In most cases, also these extension problems have been resolved:

\begin{itemize}[leftmargin=0.5cm]
\item For $d$ even, $\bP_{d+1}$ vanishes and the map to $\Theta_{d}\xra{[-]} \coker(J)_{d}$ is an isomorphism as long as $d\neq 2^{k}-2$ for $k> 7$, so in these cases there are no extension problems.
\item For $d=2^{k}-2$ with $k\le6$, we have an exact sequence $0\ra \Theta_{d}\ra \coker(J)_d\ra \bfZ/2\ra0$ which admits a splitting since in these dimensions $\coker(J)_d$ is known to be annihilated by $2$ (see e.g.\,the table \cite[Table 1]{IsaksenWangXu}), so $ \Theta_{d}\cong\coker(J)_d\oplus\bfZ/2$. For $k=7$ the question whether the map $\coker(J)_{126}\ra\bfZ/2$ is split surjective (rather than just surjective which is open too; see above) is known as the \emph{strong Kervaire invariant one problem}.
\item For $d\equiv 3\pmod 4$, the map $\Theta_{d}\ra \coker(J)_d$ is split surjective by \cite[Theorem 1.3]{BrumfielI} or \cite[Theorem 5]{Frank}, so $\Theta_{d}\cong \bP_{d+1}\oplus \coker(J)_{d}$.
\item For $d\equiv 1\pmod 4$, the map $\Theta_{d}\ra \coker(J)_d$ is split surjective if $d$ is not of the form $2^k-3$ for some $k\ge1$ by \cite[Theorem 1.2]{BrumfielII} and \cite[Theorem 1.1]{BrumfielIII}, so $\Theta_{d}\cong \bfZ/2\oplus \coker(J)_{d}$.
\end{itemize}

\subsubsection{Collar twists of homotopy spheres and the Milnor--Munkres--Novikov pairing}We begin the discussion of  collar twists of homotopy spheres with a general observation: if $M$ is a closed oriented manifold of dimension $d\ge5$ and $\Sigma\in\Theta_d$ is a homotopy sphere, then writing $\overline{\Sigma}\in\Theta_d$ for the inverse sphere obtained by reversing the orientation, the maps
\[
\def\arraystretch{1.5}
\begin{array}{c@{\hskip 0.05cm} l@{\hskip 0.05cm} c@{\hskip 0.1cm} l@{\hskip 0.1cm} l@{\hskip 0cm} l@{\hskip 0cm}}
(-)\natural\id_{\Sigma^\circ}&\colon&  \BDiff_\partial(M^\circ)&\lra&\BDiff_\partial(M^\circ\natural\Sigma^\circ)=\BDiff_\partial((M\sharp\Sigma)^\circ)\\
(-)\natural\id_{\overline{\Sigma}^\circ}&\colon&\BDiff_\partial((M\sharp\Sigma)^\circ)&\lra&\BDiff_\partial((M\sharp\Sigma)^\circ\natural\overline{\Sigma}^\circ)=\BDiff_\partial(M^\circ).\\
\end{array}
\]
are inverse homotopy equivalences, so in particular induce an isomorphism $\pi_0\,\Diff_\partial(M^\circ)\cong \pi_0\,\Diff_\partial((M\sharp\Sigma)^\circ)$ on fundamental groups. For $M=S^d$, combining the latter with the usual isomorphism $\pi_0\,\Diff_\partial(D^d)\cong\Theta_{d+1}$ given by gluing together two copies $D^{d+1}$ along their boundary via diffeomorphisms of $S^{d}$ supported on a hemisphere results in a chain of isomorphisms
\[\pi_0\,\Diff_\partial(\Sigma^\circ)\cong\pi_0\,\Diff_\partial(D^d)\cong\Theta_{d+1}
\]
We write \[T_\Sigma\in\Theta_{d+1}\]
for the image of the collar twist $t_{\Sigma}\in  \pi_0\,\Diff_\partial(\Sigma^\circ)$ under these isomorphisms. This defines a set-theoretical function $T_{(-)}\colon \Theta_d\ra \Theta_{d+1}$ which can be rephrased (see \cref{prop:Kreck-Levine} below) in terms of a well-known construction in the study of homotopy spheres, namely the bilinear Milnor--Munkres--Novikov pairing (see e.g.\,\cite{Bredon}) $\pi_k\,\bfS\otimes \Theta_d\ra \Theta_{k+d}$ for $k<d-1$. The latter is related to the multiplication in the stable homotopy groups of spheres by a commutative diagram
\begin{equation}\label{equ:compatility}
\begin{tikzcd}
\pi_k\,\bfS\otimes \Theta_d\rar{(-)\cdot(-)}\dar[swap]{\id_{\pi_k\bfS}\otimes[-]}&[20pt] \Theta_{k+d}\dar{[-]}\\
\pi_k\,\bfS\otimes\coker(J)_d\rar{(-)\cdot(-)}&\coker(J)_{k+d}
\end{tikzcd}\quad\text{ for }k<d-1
\end{equation}
with bottom horizontal map induced by the multiplication on the stable stems, using that products of elements in $\im(J)_k$ and $\pi_d\,\bfS$ contained in $\im(J)_{k+d}$ if $k<d-1$ (see p.\,442 of loc.cit.).

\begin{prop}[Kreck, Levine]\label{prop:Kreck-Levine}We have $T_\Sigma = \eta \cdot \Sigma$ where $\eta\in\pi_1\,\bfS\cong\bfZ/2$ is the generator.\end{prop}

\begin{proof} Levine writes $\gamma(\Sigma)\in\Theta_{d+1}$ for $T_\Sigma\in\Theta_{d+1}$ \cite[p.\,245-246]{Levine} and Kreck writes $\Sigma_\Sigma\in\Theta_{d+1}$ for it \cite[p.\,646]{KreckIsotopy}. For even $d$, the claim is \cite[Lemma 3 c)]{KreckIsotopy}. For odd $d$, the subgroup $\mathrm{bP}_{d+2}\le \Theta_{d+1}$ is trivial, so it suffices to show the claimed equality after passing to $\coker(J)_{d+1}$ (see \cref{sec:classification}). The latter follows from \cite[Corollary 4]{Levine} using that Levine's subgroup $I_1(\Sigma) \subset \Theta_{d+1}$ is generated by $\gamma(\Sigma)\in\Theta_{d+1}$ by definition; see p.\,246 loc.cit..
\end{proof}

\begin{rem}\cref{prop:Kreck-Levine} has immediate consequences for collar twists of homotopy spheres. For example, since $\eta$ is $2$-torsion and the Milnor--Munkres--Novikov pairing is bilinear, the sphere $T_\Sigma=\eta\cdot \Sigma$ is trivial if $\Sigma\in\Theta_d$ has odd order, so the collar twist of $\Sigma$ is in these cases trivial too.
\end{rem}

The combination of \cref{prop:Kreck-Levine}, the classification of homotopy spheres as recalled in \cref{sec:classification}, and the diagram \eqref{equ:compatility} allows one to reduce most questions on collar twists of exotic spheres to questions in stable homotopy theory. As an example of this principle, we rephrase the condition featuring in the statements of Theorem \ref{bigthm:splitting} and \ref{bigthm:mcg-homotopy-tori} (whether $T_\Sigma=\eta \cdot \Sigma\in\Theta_{d+1}$ is divisible by $2$ or not) in most cases in terms of the cokernel of the stable $J$-homomorphism:

\begin{lem}\label{lem:reduct-to-coker} If $\eta \cdot \Sigma \in \Theta_{d+1}$ is divisible by $2$, then so is $ \eta \cdot [\Sigma] \in \coker(J)_{d+1}$. The converse holds
\begin{enumerate}
\item for $d \not \equiv 4,5 \pmod 8$,
\item for $d \equiv 5\ \ \ \,\pmod 8$ for $d\neq125$, and
\item for $d \equiv 4\ \ \ \,\pmod 8$ for $d=2^{k}-4$ with $k\le 6$.
\end{enumerate}
\end{lem}

\begin{proof}
By commutativity of \eqref{equ:compatility}, the class $\eta \cdot [\Sigma]\in\coker(J)_{d+1}$ is the image of $\eta \cdot \Sigma \in \Theta_{d+1}$ under the morphism $[-]\colon \Theta_{d+1}\ra \coker(J)_d$, so if the latter is divisible by $2$, then so is the former. To prove the partial converse, we distinguish some cases and make frequent use of the classification of homotopy spheres as recalled in \cref{sec:classification}, without further reference.

\begin{itemize}[leftmargin=0.5cm]
\item For $d+1\equiv3,7\ \mathrm{mod}\ 8$, the map $[-]\colon\Theta_{d+1}\ra \coker(J)_{d+1}$ is split surjective, so $\Theta_{d+1}\cong \bP_{d+2}\oplus \coker(J)_{d+1}$. Since $\eta \cdot \Sigma$ has order two and $\bP_{d+2}$ is cyclic of order divisible by $4$, the $\bP_{d+2}$-component of the order $2$ element $\eta \cdot \Sigma$ has to be divisible by $2$, so the full element $\eta \cdot \Sigma$ is divisible by $2$ if and only if its image $\eta\cdot[\Sigma]\in \coker(J)_{d+1}$ is divisible by $2$.
\item For $d+1\equiv1\ \mathrm{mod}\ 8$, the map $[-]\colon\Theta_{d+1}\ra \coker(J)_{d+1}$ is also split surjective, so $\Theta_{d+1}\cong \bP_{d+2}\oplus \coker(J)_{d+1}$ . In this case the $\bP_{d+2}$-component of $\eta \cdot \Sigma\in\Theta_{d+1}$ turns out to vanish, which implies the result. The reason for this vanishing is that $\eta \cdot \Sigma\in\Theta_{d+1}$ is contained in the subgroup $\bSpin_{d+2}\le \Theta_{d+1}$ of homotopy spheres that bound a spin manifold \cite[\S 4 + Diagram (6)]{Lawson} and on this subgroup the $\bP_{d+2}$-component with respect to the can be computed as the image of the $f$-invariant from \cite[§3]{BrumfielII} which vanishes for $\eta \cdot \Sigma$ by \cite[Proposition 4.1]{Lawson} (this uses that the pairings denoted $\tau_{n,k}$ and $\rho_{n,k}$ in loc.cit.\, are compatible, by diagram (B) on p.~835 of loc.cit.).
\item For $d+1\equiv0,2,4,6\ \mathrm{mod}\ 8$ and $d+1\neq2^k-2$ for $k\le 7$, and for $d+1 \equiv 5 \pmod 8$ with $d+1=2^{k}-3$ for $k\le 6$ we have  $\Theta_{d+1}\cong \coker(J)_{d+1}$ and there is nothing to show.
\item For $d+1 \equiv 6 \pmod 8$ with $d+1=2^{k}-2$ for $k\le 6$ we have $\Theta_{d+1}\cong \coker(J)_{d+1}\oplus \bfZ/2$, so an element in $\Theta_{d+1}$ is divisible by $2$ if and only if this holds for its image in $\coker(J)_{d+1}$. \qedhere
\end{itemize}
\end{proof}

\begin{rem}To extend \cref{lem:reduct-to-coker} to $d+1 \equiv 5 \pmod 8$ for $d+1 \neq 2^k-3$, it would suffice to show that the $\bP_{d+2}$-component of $\eta \cdot \Sigma$ for $\Sigma\in \Theta_d$ under the splitting $\Theta_{d+1} \cong \coker(J)_{d+1} \oplus \bP_{d+2}$ recalled in \cref{sec:classification} is trivial. We do not know whether this is the case.
\end{rem}

In view of \cref{lem:reduct-to-coker}, the question whether $\eta\cdot\Sigma\in\Theta_{d+1}$ is divisible by $2$ can in many dimensions be analysed with inputs from stable homotopy theory. The following two remarks contain some applications in this direction:

\begin{rem}\label{ex:table} As $\eta$ has order two, whether $\eta\cdot\Sigma\in\Theta_{d+1}$ is divisible by $2$ or not can be tested $2$-locally. At the prime $2$, the groups $\coker(J)_d$ and multiplication by $\eta$ on them have been computed up to dimensions about $90$. The result is summarised in \cite[Figure 1]{IsaksenWangXu} where every dot represents a nontrivial element, the diagonal and vertical lines indicate that two elements are related by multiplication with $\eta$ or $2$, respectively, and the image of $J$ consists of the blue dots, apart from the blue dots in degrees $\equiv1,2 \pmod 8$. Combining this with \cref{lem:reduct-to-coker} and the classification of homotopy spheres recalled in \cref{sec:classification}, one can in most dimensions up to about $90$ determine the groups $\Theta_{d}$ and the subgroups $\smash{\Theta^{\mathrm{split}}_{d}\le \Theta_{d}}$ of those $\Sigma\in\Theta_d$ such that $\eta\cdot \Sigma$ is divisible by $2$. The result of this analysis for $d\le 19$ is recorded in \cref{tab:split} of the introduction.
\end{rem}

\begin{rem}There are also many infinite families of homotopy spheres $\Sigma\in\Theta_d$ for which one can decide whether $\eta\cdot\Sigma\in\Theta_{d+1}$ is divisible by $2$ or not. We again rely on \cref{sec:classification}.
\begin{itemize}[leftmargin=0.5cm]
\item As an infinite family of nontrivial $\Sigma\in\Theta_d$ in odd dimensions such that $\eta\cdot\Sigma$ is divisible by $2$, one may for instance take any $\Sigma\in\Theta_{d}$ for $d\equiv 1,3,7\pmod 8$ that lies in the nontrivial subgroup $\bP_{d+1}\le \Theta_{d}$. This is because $\eta\cdot \Sigma\in\Theta_{d+1}\cong \coker(J)_{d+1}$ is trivial as a result of \eqref{equ:compatility}, so it is in particular divisible by $2$. There are also examples in even dimensions: as $\bP_{8k+3}=0$, the class in $\coker(J)_{8k+2}$ of Adams' element $\mu_{8k+2}\in\pi_{8k+2}\,\bfS$ (which is nontrivial in $\coker(J)_{8k+2}$ as $\pi_{8k+2}\,\oO=0$) lifts uniquely to a homotopy sphere $\Sigma_{\mu_{8k+2}}\in\Theta_{8k+2}$. As $\eta\cdot [\mu_{8k+2}]=0\in \coker(J)_{8k+2}$ since $\eta\cdot \mu_{8k+2}\in\pi_{8k+3}\,\bfS$ is known to be contained in $\im(J)_{8k+3}$, it follows from \cref{lem:reduct-to-coker} that $\Sigma_{\mu_{8k+2}}$ is divisible by $2$.
\item As an infinite family of nontrivial $\Sigma\in\Theta_d$ in odd dimensions such that $\eta\cdot\Sigma$ is not divisible by $2$, one may take any $\Sigma\in\Theta_{8k+1}$ that maps to the class in $\coker(J)_{8k+1}$ represented by Adams' element  $\mu_{8k+1}\in\pi_{8k+1}\,\bfS$ which is known to have the property that $\eta\cdot\mu_{8k+1}\in \coker(J)_{8k+2}=\pi_{8k+2}\,\bfS$ is not divisible by $2$. In even dimensions, one may use the families of nontrivial homotopy spheres $\Sigma\in \Theta_{d}$ for $d \equiv 8 \pmod {192}$ from \cite[Proposition 2.11 (i)]{KrannichExotic} which have the property that $\eta\cdot[\Sigma]\in\coker(J)_{d}$ is nontrivial and detected in the spectrum $\mr{tmf}$ of topological modular forms (see the proof of the cited proposition). Moreover, in these dimensions $\pi_{d+1}\,\mr{tmf}$ is known to be annihilated by $2$ (see e.g.\,\cite[Figure 1.2]{Behrens}), so $\eta\cdot[\Sigma]$ is not divisible by $2$ in $\pi_{d+1}\,\mr{tmf}$ and hence neither in $\coker(J)_{d+1}$.
\end{itemize}
\end{rem}

\subsection{Collar twists of tori} The next class of manifolds for which we establish some results on their collar twisting maps are homotopy tori. For this class of manifolds, it is convenient to study the fibre sequence \eqref{equ:collar-twist-sequence} involving the collar twisting maps by comparing it to an analogous sequence for block-homeomorphisms (see e.g.\,\cite[Section 2]{HLLRW} for a discussion of block-automorphisms suitable for our needs) via a map of fibre sequences
\begin{equation}\label{diff-to-block}
\begin{tikzcd}
\BDiff_\partial(M^\circ)\rar\dar& \BDiff^+_\ast(M)\rar\dar& \BSO(d)\dar\\
\BlockBHomeo_{T_*M}(M)\rar&\BlockBHomeo^+_\ast(M)\rar&\BSTOP.
\end{tikzcd}
\end{equation}
The bottom row of this diagram deserves an explanation. To construct it, first consider the forgetful map $\BlockHomeo^+(M)\ra \hAut^+(M)$ from the space of orientation-preserving block-homeomorphisms of $M$ to the space of orientation-preserving homotopy self-equivalences. The space $\BlockHomeo^+_\ast(M)$ is defined as the homotopy pullback of this map along the inclusion map $\hAut_\ast^+(M)\ra \hAut^+(M)$ of those orientation-preserving self-equivalences of $M$ that preserve the chosen point $\ast\in M$. The delooping of the latter map is the universal $M$-fibration, so $\BlockBHomeo^+_\ast(M)$ is by construction equivalent to the total space of the universal oriented $M$-block-bundle. The right-hand map in the bottom sequence is the delooping of the map $\BlockHomeo^+_\ast(M)\rightarrow \STOP$ that takes the stable topological derivative of a block-homeomorphism of $M$ at $\ast\in M$, or equivalently, it classifies the stable vertical topological tangent bundle of the universal oriented $M$-block-bundle (see Section 2 loc.cit.), similarly to how the right-hand map of the upper sequence classifies the vertical tangent bundle of the universal oriented smooth $M$-bundle. The rightmost vertical map classifies the underlying stable Euclidean bundle of an oriented $d$-dimensional vector bundle and the middle vertical map is induced by the forgetful map $\Diff_\ast^+(M)\ra  \BlockHomeo^+_\ast(M)$. The left-hand map of the bottom sequence is \emph{defined} as the homotopy fibre inclusion of the right-hand map, or equivalently, as the delooping of the derivative map.

Note that the bottom row only depends on the underlying topological manifold of $M$, so in particular agrees for homotopy tori $M=\cT$ with the corresponding sequence of the standard torus $M=T^d$. For the latter, the middle space $\BlockBHomeo^+_\ast(M)$ has a very simple description:
\begin{lem}\label{lem:pointed-block-homeo}
For any $d\ge1$, the map \[\BlockBHomeo^+_\ast(T^d)\lra\BSL_d(\bfZ)\] induced by the action on $\oH_1(T^d)\cong\bfZ^d$ is an equivalence.
\end{lem}
\begin{proof}
As $T^d\simeq K(\bfZ^d,1)$, the analogous map $\hAut^+_\ast(T^d)\ra\SL_d(\bfZ)$ from the space of orientation homotopy self-equivalences of $M$ is an equivalence, so it suffices to show that the forgetful map $\BlockHomeo^+_\ast(T^d)\ra \hAut^+_\ast(T^d)$ is an equivalence. We will do so by proving that the right-hand map in the map of homotopy fibre sequences
\[
\begin{tikzcd}
T^d\rar\arrow[d,equal] &\BlockBHomeo^+_\ast(T^d)\rar\dar&\BlockBHomeo^+(T^d) \arrow[d]\\
 T^d\rar &\BhAut^+_\ast(T^d)\rar&\BhAut^+(T^d)
\end{tikzcd}
\]
comparing the universal $T^d$-block-bundle with the universal $T^d$-fibration, is an equivalence.
Using the action of $\SL_d(\bfZ)$ on $T^d\cong\bfR^d/\bfZ^d$ and the action of $T^d$ on itself, a diagram chase in the ladder of long exact sequences induced by this map of fibre sequences shows that the middle arrow is surjective on all homotopy groups. Injectivity on homotopy groups is equivalent to the claim that for $k\ge0$, any self-homeomorphism of $T^d\times D^k$ fixing on the boundary that is homotopic to the identity relative to the boundary is also concordant to the identity relative to the boundary. For $d\ge5$, this from the fact that the topological structure sets $\smash{\cS^{\TOP}_\partial(T^d\times D^k)}$ in the sense of surgery theory are trivial as long as $k+d\ge5$ \cite[p.\,205, Theorem C.2]{KirbySiebenmann}, but there is also a more direct proof in all dimensions  \cite{LawsonHomeomorphisms}.
\end{proof}

\begin{cor}\label{cor:collar-twist-injectivity}Let $\cT$ be a homotopy torus of dimension $d\ge1$. The collar twisting map
\[\Upsilon_{\cT} \colon \SO(d)\lra \BDiff_\partial(\cT^{\circ})\]
is injective on $\pi_k(-)$ for $k\le d-2$. In particular, $t_{\cT}\in\pi_0\,\Diff_\partial(\cT^\circ)$ is nontrivial for $d\ge3$.
\end{cor}

\begin{proof}
From the map of long exact sequences induced by the map \eqref{diff-to-block} together with the fact that the higher homotopy groups of $\BlockBHomeo^+_\ast(\cT)\simeq \BlockBHomeo^+_\ast(T^d)$ vanish as a result of \cref{lem:pointed-block-homeo}, we see that the map in question is injective on $\pi_k(-)$ if the map $\pi_k\,\SO(d)\ra \pi_k\,\TOP$ is injective. This maps factors as the stabilisation map $\SO(d)\ra \SO$ followed by the forgetful map $\SO\ra \TOP$. The latter is injective on all homotopy groups (combine \cite{BrumfielI} with \cite[p.\,246, 5.0.(1)]{KirbySiebenmann}) and the former for $k\le d-2$ by stability, so the claim follows.
\end{proof}

\begin{rem}The collar twist $t_{\cT}\in\pi_0\,\Diff_\partial(\cT^\circ)$ is also nontrivial for $d=2$ which follows for instance from \cref{lem:mcg-torus-bdy} below. Moreover, for $d=2,3$ the map $\Upsilon_\cT$ induces an isomorphism on all higher homotopy groups since the component of the identity $\Diff_*(T^d)_\id \simeq \ast$ is contractible. This is well-known for $d=2$ and follows for $d=3$ from a combination of  \cite{Hatcher} and \cite{HatcherIrreducible} using that $T^3$ is Haken.\end{rem}

\begin{rem}\label{rem:general-aspherical} Replacing $\SL_d(\bfZ)$ with $\Aut(\pi_1\,M)$, the statement of \cref{lem:pointed-block-homeo} (and thus also that of \cref{cor:collar-twist-injectivity}) holds for many other closed aspherical manifolds $M$, in particular for those of dimension $d\ge5$ whose fundamental group satisfies the Farrell--Jones conjecture and also for those of dimension $d=4$ if the fundamental group is  \emph{good} in the sense of \cite[p.~99]{FreedmanQuinn} (see e.g.\,\cite[Proposition 5.1.1]{HLLRW} for an explanation of this).
\end{rem}

\section{Mapping class groups of exotic tori and the proof of \cref{bigthm:mcg-homotopy-tori}}
Equipped with the results on collar twists from the previous section, we turn towards studying the mapping class groups of homotopy tori of the form $T^d \sharp \Sigma$.

\subsection{Central extensions of special linear groups}\label{sec:extensions} The strategy will be to relate the mapping class groups of homotopy tori to well-known central extensions of special linear groups. We first recall these extensions and discuss some of their properties. The universal cover of the stable special linear group over the reals $\SL(\bfR) = \colim_d\, \SL_d(\bfR)$ gives a central extension
\[0\lra \bfZ/2\lra \overline{\SL}(\bfR)\lra \SL(\bfR)\lra 0\]
which we may pull back along the lattice inclusion $\SL_d(\bfZ)\le \SL(\bfR)$ to a central extension
\begin{equation}\label{equ:sl-bar}0\lra \bfZ/2\lra \overline{\SL}_d(\bfZ)\lra \SL_d(\bfZ)\lra 0\end{equation} for $d\ge1$. For $d=2$, also a different central extension will play a role, namely the pullback
\begin{equation}\label{equ:sl-tilde}0\lra \bfZ\lra \widetilde{\SL}_2(\bfZ)\lra \SL_2(\bfZ)\lra 0\end{equation}
along the inclusion $\SL_2(\bfZ)\le \SL_2(\bfR)$ of the universal cover central extension
\[0\lra \bfZ\lra \widetilde{\SL}_2(\bfR)\lra \SL_2(\bfR)\lra 0.\]
Since the inclusion map $ \SL_d(\bfR)\ra  \SL_{d+1}(\bfR)$ is surjective on fundamental groups for $d\ge2$, the extension \eqref{equ:sl-bar} agrees with the pushout of \eqref{equ:sl-tilde} along the quotient map $\bfZ\ra\bfZ/2$. Everything we need to know about these extensions, together with some useful information on the low-degree (co)homology of $\SL_d(\bfZ)$ is summarised in the following lemma.

\begin{lem}\label{lem:low-degree-homology}\ \begin{enumerate}
		\item\label{sl-i} The first two homology groups of $\SL_d(\bfZ)$ are given by the following table
		\[\begin{tabular}{l|l|l}
			\toprule
			$d$ & $\oH_1(\SL_d(\bfZ);\bfZ)$ & $\oH_2(\SL_d(\bfZ);\bfZ)$\\
			\midrule
			$2$ & $\bfZ/12$ & $0$\\
			$3$& $0$ & $\bfZ/2\oplus \bfZ/2$  \\
			$4$& $0$ & $\bfZ/2\oplus \bfZ/2$ \\
			$\ge5$ & $0$ & $\bfZ/2$ \\
			\bottomrule
		\end{tabular}\]
		\item\label{sl-ii} The map \[\oH_2(\SL_d(\bfZ);\bfZ)\lra \oH_2(\SL_{d+1}(\bfZ);\bfZ)\] induced by stabilisation is nontrivial for $d\ge3$. For $d=3$ its image has order two.
			\item \label{sl-iii}The extension \eqref{equ:sl-tilde} is classified by a generator of \[\oH^2(\SL_2(\bfZ);\bfZ) \cong \bfZ/12.\]
		\item\label{sl-iv} The extension \eqref{equ:sl-bar} is nontrivial for $d\ge2$. For $d\ge5$, it is classified by the generator of \[\oH^2(\SL_d(\bfZ);\bfZ/2) \cong \bfZ/2.\]
		\end{enumerate}
\end{lem}

\begin{proof}
That the abelianisation $\oH_1(\SL_d(\bfZ);\bfZ)$ is cyclic of order $12$ for $d=2$ can be read off from the standard presentation of $\SL_2(\bfZ)$ (see e.g.\,\cite[Corollary 10.5]{MilnorAlgebraic}), and the fact that $\oH_1(\SL_d(\bfZ);\bfZ)$ vanishes for $d\ge3$ follows for instance from \cite[Corollary 10.3]{MilnorAlgebraic} together with the observation that for $d\ge3$ every elementary square matrix can be written as a commutator of two elementary matrices. The computation of $\oH_2(\SL_d(\bfZ);\bfZ)$ for $d=2$ follows from the isomorphism $\SL_2(\bfZ) \cong \bfZ/4 \ast_{\bfZ/2} \bfZ/6$ \cite[1.5.3]{Serre} and the Mayer--Vietoris sequence for group homology of amalgamated products \cite[Corollary II.7.7]{Brown} (c.f.~Exercise 3 on p.~52 of loc.cit.), for $d=3,4$ from \cite{vanderKallen}, and for $d\ge5$ from \cite[Corollary 5.8, Remark on p.\,48, Theorem 10.1]{MilnorAlgebraic}. Using these calculations, the computations $\oH^2(\SL_2(\bfZ);\bfZ) \cong \bfZ/12$ and $\oH^2(\SL_d(\bfZ);\bfZ/2) \cong \bfZ/2$ for $d\ge5$ implicitly claimed in \ref{sl-iii} and \ref{sl-iv} follow from the universal coefficient theorem. The latter also implies the part of \ref{sl-ii} for $d\ge3$ once we show \ref{sl-iv}. The claim on the stabilising map for $d=3$ can be proved via the arguments in \cite{vanderKallen}.

To prove \ref{sl-iii}, we use the general fact that for a given central extension
$0\ra A\ra E\ra G\ra 0$,
the Serre spectral sequence induces an exact sequence
$0\ra\oH^1(G;A)\ra \oH^1(E;A)\ra \oH^1(A;A)\ra\oH^2(G;A)\ra \oH^2(E;A)$.
The identity map induces a preferred class in $\oH^1(A;A)$ and its image in $\oH^2(G;A)$ is the class that classifies the given extension. Applying this to the extension \eqref{equ:sl-tilde}, we see that in order to show that this extension generates $\oH^2(\SL_2(\bfZ);\bfZ) $ it suffices to show that $\oH^2(\smash{\widetilde{\SL}}_2(\bfZ);\bfZ) $ vanishes. Now $\smash{\widetilde{\SL}}_2(\bfZ)$ agrees up to isomorphism with the braid group $B_3$ on three strands (see e.g.\,\cite[p.\,83]{MilnorAlgebraic}), so the claim follows from the universal coefficient theorem and the facts that $\oH_1(B_3;\bfZ)\cong\bfZ$ and $\oH_2(B_3;\bfZ) \cong \bfZ/2$ \cite[p.~32]{Arnold}.

For \ref{sl-iv}, we use the universal coefficient theorem to see that $\oH^2(\SL_2(\bfZ);\bfZ/2)\cong\bfZ/2$ is surjected upon by the map
$\oH^2(\SL_d(\bfZ);\bfZ)\ra \oH^2(\SL_d(\bfZ);\bfZ/2)$ induced by reduction modulo $2$, so it follows from \ref{sl-iii} that the extension
\eqref{equ:sl-bar} is nontrivial for $d=2$ and hence also for all higher values of $d$ since the former is the pullback of the latter along the inclusion $\SL_2(\bfZ)\ra \SL_d(\bfZ)$. As $\oH^2(\SL_d(\bfZ);\bfZ/2) \cong \bfZ/2$ has only a single nontrivial element for $d \geq 5$, this gives \ref{sl-iv}.\end{proof}

\subsection{Mapping class groups of homotopy tori and \cref{bigthm:mcg-homotopy-tori}} We now determine the mapping class groups of homotopy tori of the form $\cT=T^d\sharp\Sigma$. The argument has three steps.
\begin{enumerate}[leftmargin=1.3cm]
	\item[\ref{step:disc-vs-point}]Determine $\pi_0\,\Diff^+_*(T^d \sharp \Sigma)$ in terms of $\pi_0\,\Diff_\partial((T^d \sharp \Sigma)^\circ)$.
	\item[\ref{step:mcg-point}] Determine $\pi_0\,\Diff^+_*(T^d \sharp \Sigma)$.
	\item[\ref{step:mcg-no-point}] Determine $\pi_0\,\Diff^+(T^d \sharp \Sigma)$.
\end{enumerate}
Throughout this section, we abbreviate  $T^{d,\circ}\coloneq (T^d)^{\circ}$, fix a basis of $\oH_1(T^{d,\circ})\cong\bfZ^d$, and use the bases for the first homology groups of $T^d$, $T^{d,\circ}\natural\Sigma^\circ=(T^d\sharp\Sigma)^{\circ}$, and $T^d\sharp\Sigma$ that are induced by the chosen basis of $\oH_1(T^{d,\circ})$.

\renewcommand\thesubsubsection{Step $\circled{\arabic{subsubsection}}$}


\subsubsection{Fixing a disc or a point}\label{step:disc-vs-point}
We first determine the group $\pi_0\,\Diff_\partial(\cT^{\circ})$ in terms of the group $\pi_0\,\Diff^+_\ast(\cT)$. This step works for general homotopy tori $\cT$, not just those of the form $T^d\sharp\Sigma$.

\begin{lem}\label{lem:mcg-torus-bdy} For a homotopy $d$-torus $\cT$, there are pullback squares
	\[\begin{tikzcd}[column sep=0.7cm, row sep=0.5cm] \pi_0\,\Diff_\partial(\cT^{\circ}) \rar {\ext}\dar{\cong} & \pi_0\,\Diff^+_\ast(\cT) \dar{\cong} \\
	\widetilde{\SL}_d(\bfZ) \rar & \SL_d(\bfZ),\end{tikzcd}\text{for $d=2$ and}\begin{tikzcd}[column sep=0.7cm, row sep=0.5cm] \pi_0\,\Diff_\partial(\cT^{\circ}) \rar {\ext} \dar & \pi_0\,\Diff^+_\ast(\cT) \dar \\
	\overline{\SL}_d(\bfZ) \rar & \SL_d(\bfZ)\end{tikzcd}\quad\text{for $d\ge3$}\]
\end{lem}

\begin{proof}If $d=2$, then $\cT$ is the standard $2$-torus $T^2$ for which the claimed square is well-known (for a reference, compare the standard presentations of $\pi_0\,\Diff_\partial(T^{d,\circ})$ and $\widetilde{\SL}_d(\bfZ)$ e.g.\,in \cite[p.82--83]{MilnorAlgebraic} and \cite[Section 5]{Korkmaz}). For $d\ge3$,  we consider the map of central extensions
\begin{equation}\label{equ:pullback-diff-to-block}
\begin{tikzcd}
0\rar&\pi_1\,\SO(d)\rar\arrow[d,"\cong"]&\pi_0\,\Diff_\partial(\cT^\circ) \rar\dar&\pi_0\,\Diff^+_\ast(\cT) \rar\dar& 0\\
0\rar&\pi_1\,\STOP\rar&\pi_0\,\BlockHomeo_{T_\ast T^d}(T^d) \rar&\pi_0\,\BlockHomeo^+_\ast(T^{d}) \rar& 0
\end{tikzcd}
\end{equation}
induced by \eqref{diff-to-block} for $M=\cT$, using that the bottom sequence only depends on the underlying topological manifold. Exactness at $\pi_1\,\STOP$ follows from the fact that $\pi_1\,\BlockHomeo^+_\ast(T^{d})=0$ by \cref{lem:pointed-block-homeo} and exactness at $\pi_1\,\SO(d)$ follows from exactness at $\pi_1\,\STOP$. \cref{lem:pointed-block-homeo} also shows that the homology action map $\pi_0\,\BlockHomeo^+_\ast(T^{d})\ra \SL_d(\bfZ)$ is an isomorphism, so we are left to show that the bottom extension is isomorphic to $\smash{0\ra\bfZ/2\ra\overline{\SL}_d(\bfZ)\ra\SL_d(\bfZ)\ra0}$. It suffices to show this for large enough $d$, since the bottom extension in dimension $d$ maps by taking products with $S^1$ to the corresponding extension in dimension $d+1$ (which have both kernel $\pi_1\STOP$), so the extension for $d$ is the pullback of the extension for $d+1$ along the inclusion $\SL_d(\bfZ)\ra \SL_{d+1}(\bfZ)$. We may thus assume $d\ge5$ in which case there is a single nontrivial central extension of $\SL_d(\bfZ)$ by $\bfZ/2$ (see \cref{lem:low-degree-homology}), so we only need to exclude that the bottom extension in \eqref{equ:pullback-diff-to-block} is trivial. To show this, we consider \eqref{equ:pullback-diff-to-block} for $\cT=T^d$ and extend it to the top as
\[
\begin{tikzcd}
0\rar&\pi_1\,\SO(2)\rar\arrow[d,"\inc_*"]&\pi_0\,\Diff_\partial(T^{2,\circ}) \rar\dar&\pi_0\,\Diff^+_\ast(T^2) \rar\dar& 0\\
0\rar&\pi_1\,\SO(d)\rar\arrow[d,"\cong"]&\pi_0\,\Diff_\partial(T^{d,\circ}) \rar\dar&\pi_0\,\Diff^+_\ast(T^d) \rar\dar& 0\\
0\rar&\pi_1\,\STOP\rar&\pi_0\,\BlockHomeo_{T_\ast T^d}(T^d) \rar&\pi_0\,\BlockHomeo^+_\ast(T^{d}) \rar& 0.
\end{tikzcd}
\]
where the top middle vertical map is induced by taking products with $\id_{T^{d-2}}$ followed by restriction, commutativity of the left upper square follows by an application of \cref{prop:ct-ctdsum-product} \ref{enum:ct-product} to $M=T^2$ and $N=T^{d-2}$, and the right upper square is induced by the commutativity of the left upper square. We will show that the bottom extension is nontrivial by showing that its pullback along the composition $\pi_0\,\Diff^+_\ast(T^2) \ra\pi_0\,\BlockHomeo^+_\ast(T^{d})$ is nontrivial. This composition is isomorphic to the inclusion $\SL_2(\bfZ)\ra \SL_d(\bfZ)$ and the composition $\pi_1\,\SO(2)\ra \pi_1\,\STOP$ to the quotient map $\bfZ\ra \bfZ/2$, so it follows that the pullback in question is isomorphic to the mod $2$ reduction of the extension $\smash{0\ra \bfZ\ra \widetilde{\SL}_2(\bfZ)\ra \SL_2(\bfZ)\ra 0}$, i.e.\,the extension $\smash{0\ra \bfZ/2\ra \overline{\SL}_2(\bfZ)\ra \SL_2(\bfZ)\ra 0}$. The latter is nontrivial by \cref{lem:low-degree-homology} \ref{sl-iv}.\end{proof}

\subsubsection{The pointed mapping class group of $T^d\sharp\Sigma$} \label{step:mcg-point} Next, we determine the group $\pi_0\,\Diff^+_*(T^d \sharp \Sigma)$. For $\Sigma=S^d$ the evaluation fibration $\Diff^+(T^d)\ra T^d$ whose fibre is $\Diff^+_*(T^d )$ has a splitting given by the standard action of $T^d$ on itself, so the long exact sequence in homotopy groups induces the first out of two isomorphisms
\vspace{-0.1cm}
\[\pi_0\,\Diff^+_*(T^d )\cong \pi_0\,\Diff^+(T^d)\overset{d\ge6}{\cong} \SL_d(\bfZ)\ltimes\Omega;\]
the second isomorphism was explained in the introduction. Combining this with \cref{lem:mcg-torus-bdy} for $\cT=T^d$, we obtain an isomorphism
\begin{equation}\label{equ:mcg-with-bdy}\pi_0\,\Diff_\partial(T^{d,\circ})\cong
 \overline{\SL}_d(\bfZ)\ltimes \Omega\quad\text{for }d\ge6.\end{equation} Now recall that the collar twist $t_{T^d}\in \pi_0\,\Diff_\partial(T^{d,\circ})$ generates the kernel of the map to $\pi_0\,\Diff^+_\ast(T^{d})$ so it corresponds under the isomorphism \eqref{equ:mcg-with-bdy} to the element $(t_d,0)$ where $t_d\in \overline{\SL}_d(\bfZ)$ is the central element that generates the kernel of the map to $\SL_d(\bfZ)$. The composition\vspace{-0.1cm} \begin{equation}\label{equ:insert-htp-sphere}
 \Theta_{d+1}\cong\pi_0\,\Diff_\partial(D^d)\xrightarrow{\id_{T^{d,\circ}}\natural(-)}\pi_0\,\Diff_\partial(T^{d,\circ}),\end{equation} which we abbreviate by $\iota_d\colon\Theta_{d+1}\ra \pi_0\,\Diff_\partial(T^{d,\circ})$, can be identified in terms of \eqref{equ:mcg-with-bdy} as follows:
\begin{lem}\label{lem:plugging-in-disc} With respect to the isomorphism \eqref{equ:mcg-with-bdy}, the composition \eqref{equ:insert-htp-sphere} agrees  with the inclusion $\Theta_{d+1}\le \overline{\SL}_d(\bfZ)\ltimes\Omega$  given by the $(j=0)$-summand in \eqref{equ:torelli}.
\end{lem}
\begin{proof}
By an application of \cref{lem:mcg-torus-bdy}, this would follow from showing the analogous statement for $\pi_0\,\Diff^+_*(T^{d})\cong\SL_d(\bfZ)\ltimes\Omega$ instead of for $\pi_0\,\Diff_\partial(T^{d,\circ})$ once we know that the postcomposition of \eqref{equ:insert-htp-sphere} with the isomorphism \eqref{equ:mcg-with-bdy} has image in $\Omega\le  \overline{\SL}_d(\bfZ)\ltimes\Omega$. That the statement holds in $\pi_0\,\Diff^+_*(T^{d})\cong \pi_0\,\Diff^+(T^{d})$ follows from \cite[p.\,9, Remark (5)]{Hatcher}, so it suffices to show that \eqref{equ:insert-htp-sphere} lands in the subgroup $\Omega$. To show that, note that diffeomorphisms in the image of \eqref{equ:insert-htp-sphere}  are topologically isotopic to the identity since $\Homeo_\partial(D^d)$ is contractible by the Alexander trick. In particular, \eqref{equ:insert-htp-sphere} lands in the kernel of the forgetful map $\pi_0\Diff_\partial(T^{d,\circ})\ra \pi_0\,\BlockHomeo_{T_\ast T^d}(T^d)$ which agrees via the isomorphism $\pi_0\,\Diff^+_*(T^d)\cong\pi_0\,\Diff^+(T^d)\cong\SL_d(\bfZ)\ltimes\Omega$ precisely with the subgroup $\Omega$ (see the proof of \cref{lem:mcg-torus-bdy}), so the claim follows.
\end{proof}

Given a homotopy sphere $\Sigma\in\Theta_d$ and $d\ge7$, we have the isomorphism discussed in
\cref{section:ct-homotopy-spheres}\vspace{-0.1cm}
\begin{equation}\label{eqn:bdy-std-homotopy}
	\pi_0\,\Diff_\partial(T^{d,\circ})\overset{(-)\natural\id_{\Sigma^\circ}}{\cong}\pi_0\,\Diff_\partial((T^d\sharp\Sigma)^{\circ}),
\end{equation}
so from the exact sequence \eqref{equ:ct-exact-sequence}, we see that $\pi_0\,\Diff_*(T^d\sharp\Sigma)$ is isomorphic to the quotient of $\smash{\pi_0\,\Diff_\partial(T^{d,\circ})\cong \overline{\SL}_d(\bfZ)\ltimes\Omega}$ by the central subgroup generated by the preimage of the collar twist $t_{T^d\sharp\Sigma}\in \pi_0\,\Diff_\partial((T^d\sharp\Sigma)^{\circ})$ under \eqref{eqn:bdy-std-homotopy}.

\begin{lem}\label{lem:pointed-answer-complicated}The preimage of $\smash{t_{T^d\sharp\Sigma}\in \pi_0\Diff_\partial((T^d\sharp\Sigma)^{\circ})}$ in $ \overline{\SL}_d(\bfZ)\ltimes \Omega$ under the combined isomorphisms \eqref{eqn:bdy-std-homotopy} and \eqref{equ:mcg-with-bdy} is $(t_d,\eta\cdot\Sigma)\in \overline{\SL}_d(\bfZ)\ltimes \Omega$. Consequently, we have an isomorphism
	\begin{equation}\label{equ:algebraic-description}
		\pi_0\,\Diff^+_\ast(T^d\sharp\Sigma)\cong \big( \overline{\SL}_d(\bfZ)\ltimes \Omega\big)/\langle (t_d,\eta\cdot\Sigma)\rangle.
	\end{equation}
	for $d\ge7$ which is compatible with the homomorphisms to $\SL_d(\bfZ)$.
\end{lem}

\begin{proof}We already explained how the second part follows from the first. To prove the first, we use the relation $\smash{t_{T^d\sharp\Sigma}=t_{T^d}\natural\id_{\Sigma^\circ}+\id_{T^{d,\circ}}\natural t_{\Sigma}}$ in $\pi_0\,\Diff_\partial((T^d\sharp\Sigma)^{\circ})$ ensured by \cref{prop:ct-ctdsum-product} \ref{enum:ct-ctdsum}, using which we express the element in question in $\pi_0\,\Diff_\partial(T^{d,\circ})$ as
\[\big(t_{T^d}\natural\id_{\Sigma^\circ}+\id_{T^{d,\circ}}\natural t_{\Sigma}\big)\natural\id_{\overline{\Sigma}^\circ}=t_{T^d}+\id_{T^{d,\circ}}\natural (t_{\Sigma}\natural\id_{\overline{\Sigma}^\circ})=t_{T^d}+\iota_d(\eta\cdot\Sigma)\in\pi_0\,\Diff_\partial(T^{d,\circ}).\]
Here we used the equality $T_\Sigma=\eta\cdot\Sigma$ from \cref{prop:Kreck-Levine} and the definition of $T_\Sigma$ from \cref{section:ct-homotopy-spheres}. By the discussion above, $t_{T^d}$ and $\iota_d(\eta\cdot\Sigma)$ corresponds under the isomorphism \eqref{equ:mcg-with-bdy} to the elements $(t_d,0)$ and $(0,\eta\cdot\Sigma)$ in $\smash{\overline{\SL}_d(\bfZ)\ltimes \Omega}$, so the element we are looking for is indeed $(t_d,\eta\cdot\Sigma)$.\end{proof}

The quotient of $\overline{\SL}_d(\bfZ)\ltimes\Omega$ appearing in \cref{lem:pointed-answer-complicated} can be further simplified:

\begin{lem}\label{lem:explicit-quotient}There is an isomorphism of groups
\[\big( \overline{\SL}_d(\bfZ)\ltimes \Omega\big)/\langle (t_d,\eta\cdot\Sigma)\rangle\cong \begin{cases}
 \overline{\SL}_d(\bfZ)\ltimes\big(\Omega/\langle \eta\cdot\Sigma\rangle\big)&\text{ if }
 \eta\cdot\Sigma\in\Theta_{d+1}\text{is not divisible by }2\\
 \SL_d(\bfZ)\ltimes \Omega&\text{ if }\eta\cdot\Sigma\in\Theta_{d+1}\text{is divisible by }2\\
\end{cases}\]
that is compatible with the homomorphisms to $\SL_d(\bfZ)$.\end{lem}

\begin{proof}
Since the element $\eta\cdot\Sigma$ of the finite abelian group $\Theta_{d+1}$ is of order $2$, it is not divisible by $2$ if and only if it generates a direct $\bfZ/2$-summand. We first assume that this is the case, so $\Theta_{d+1}\cong(\Theta_{d+1}/\langle\eta\cdot\Sigma\rangle)\oplus \bfZ/2$. Writing $\Omega'\le \Omega$ for the $\SL_d(\bfZ)$-invariant subgroup complementary to the central summand $\Theta_{d+1}\le \Omega$ in \eqref{equ:torelli}, we have $\smash{\overline{\SL}_d(\bfZ)\ltimes \Omega}\cong\smash{\overline{\SL}_d(\bfZ)\ltimes (\Omega'\oplus (\Theta_d/\langle\eta\cdot\Sigma\rangle)\oplus \bfZ/2)}$. The latter admits an epimorphism to $\smash{\overline{\SL}_d(\bfZ)\ltimes (\Omega'\oplus (\Theta_d/\langle\eta\cdot\Sigma\rangle))}$ given by sending $(A,a+\Sigma'+x)$ to $(x\cdot t_d\cdot A,a+\Sigma')$. This is well-defined since the central element  $\smash{t_d\in \overline{\SL}_d(\bfZ)}$ has order $2$ and acts trivially on $\Omega$ since the action factors by construction through $\SL_d(\bfZ)$. The kernel of this epimorphism is the subgroup generated by $(t_d,\eta\cdot \Sigma)$, so we obtain an isomorphism between $\smash{\big( \overline{\SL}_d(\bfZ)\ltimes \Omega\big)/\langle (t_d,\eta\cdot\Sigma)\rangle}$ and $\overline{\SL}_d(\bfZ)\ltimes \big(\Omega/\langle\eta\cdot\Sigma\rangle)\big)$, as claimed.

Now assume that $\eta\cdot\Sigma\in\Theta_{d+1}$ does not generate a direct summand, so is divisible by $2$. We have a (non-central) extension
\begin{equation}\label{equ:extension}0\lra \big(\Omega\oplus \bfZ/2\big)/\langle \eta\cdot\Sigma+[1]\rangle\lra \big(\overline{\SL}_d(\bfZ)\ltimes \Omega\big)/\langle (t_d,\eta\cdot\Sigma)\rangle \lra \SL_d(\bfZ)\lra0.\end{equation}
As $ \eta\cdot\Sigma$ has order $2$, the $\SL_d(\bfZ)$-equivariant map from $\Omega=\Omega\oplus0$ into the kernel of \eqref{equ:extension} induced by inclusion is an isomorphism, so in order to show that $\smash{\big( \overline{\SL}_d(\bfZ)\ltimes \Omega\big)/\langle (t_d,\eta\cdot\Sigma)\rangle}$ is isomorphic to $ \SL_d(\bfZ)\ltimes \Omega$ it suffices to show that \eqref{equ:extension} splits. Writing $\Omega=\Theta_{d+1}\oplus\Omega'$ as in the previous case, the extension \eqref{equ:extension} is by construction the sum of the trivial $\SL_d(\bfZ)$-extension by the $\SL_d(\bfZ)$-module $\Omega'$ with the central extension classified by the image of the unique nontrivial element in $\oH^2(\SL_d(\bfZ);\bfZ/2)$ under the composition
$\oH^2(\SL_d(\bfZ);\bfZ/2)\ra\oH^2(\SL_d(\bfZ);\bfZ/2\oplus\Theta_{d+1})\ra \oH^2(\SL_d(\bfZ);(\bfZ/2\oplus\Theta_{d+1})/(t_d,\eta\cdot\Sigma))$ induced by the inclusion and quotient maps of coefficients, so it suffices to show that this image is trivial. From the universal coefficient theorem and the computations in \cref{lem:low-degree-homology}, we see that for any abelian group, the map $\oH^2(\SL_d(\bfZ);A)\ra \oH^2(\SL_d(\bfZ);A/2)$ induced by reducing modulo $2$ is an isomorphism, so to show that the class in question is trivial, it suffices to do so after reducing modulo $2$. The latter follows by noting that the composition of $\SL_d(\bfZ)$-modules $\bfZ/2\subset \bfZ/2\oplus\Theta_{d+1}\ra \big(\bfZ/2\oplus\Theta_{d+1})/(t_d,\eta\cdot\Sigma)\big)$ is trivial after passing to $\big(\bfZ/2\oplus\Theta_{d+1})/(t_d,\eta\cdot\Sigma)\big)/2$ since $\eta\cdot\Sigma$ vanishes in $\Theta_{d+1}/2$ by assumption.
\end{proof}

\subsubsection{Fixing a point or not}\label{step:mcg-no-point} Using the description of   $\pi_0\,\Diff^+_*(T^d \sharp \Sigma)$ from the previous step, we are now in the position to determine $\pi_0\,\Diff^+(T^d \sharp \Sigma)$. In view of the fibration sequence
\begin{equation}\label{equ:pp-fibration}T^d \sharp \Sigma \lra \BDiff^+_*(T^d \sharp \Sigma) \lra \BDiff^+(T^d \sharp \Sigma)\end{equation}
this amounts to understanding the image of the ``point-pushing'' homomorphism $\pp \colon \pi_1\,T^d \sharp \Sigma \to \pi_0\,\Diff^+_*(T^d \sharp \Sigma)$. For $\Sigma = S^d$, the image is trivial a result of the action of $T^d$ on itself (see the beginning of \ref{step:mcg-point}), but for any other homotopy sphere such an action is not available and it in fact follows from the following lemma that the image is never trivial.

\begin{lem}\label{lem:point-pushing} For $d\ge7$ and $\Sigma\in \Theta_d$, the map
\[\bfZ^d= \pi_1\, T^d \sharp \Sigma\xlra{\pp}\pi_0\,\Diff^+_*(T^d \sharp \Sigma)\]
 agrees with the composition
\[\bfZ^d\xlra{(-)\otimes\Sigma }\bfZ^d\otimes\Theta_d \leq\overline{\SL}_d(\bfZ)\ltimes \Omega \cong  \pi_0\,\Diff_\partial(T^{d,\circ})\overset{(-)\natural \id_{\Sigma^\circ}}{\cong}\pi_0\,\Diff_\partial((T^d \sharp \Sigma)^\circ)\xlra{\ext} \pi_0\,\Diff^+_*(T^d \sharp \Sigma)\]
involving the isomorphism $\pi_0\,\Diff_\partial((T^d \sharp \Sigma)^\circ)\cong \overline{\SL}_d(\bfZ)\ltimes \Omega$ from \eqref{equ:mcg-with-bdy}.
\end{lem}

\begin{proof}It suffices to show that the two compositions agree on the first standard basis vector $e_1\in\bfZ^d$, since both compositions are $\pi_0\,\Diff_\partial(T^{d,\circ})$-equivariant (the action on the source is through $\SL_d(\bfZ)$ and the action on the target is by conjugation after extending diffeomorphisms from $T^{d,\circ}$ to $T^d\sharp\Sigma$ by the identity) and the orbit of $e_1\in\bfZ^d$ under the $\SL_d(\bfZ)$-action spans $\bfZ^d$.

We work with the following model of $T^d\sharp\Sigma$: view $T^{d-1}$ as $\bfR^{d-1}/\bfZ^{d-1}$, choose an orientation-preserving embedding $\iota \colon D^{d-1} \hookrightarrow T^{d-1} = \bfR^{d-1}/\bfZ^{d-1}$ disjoint from the origin $[0]\in \bfR^{d-1}/\bfZ^{d-1}=T^{d-1}$ and a representative $f_\Sigma \in \Diff_\partial(D^{d-1})$ of $\Sigma \in \pi_0\,\Diff_\partial(D^{d-1}) \cong \Theta_d$, extend $f_\Sigma$ by the identity to a diffeomorphism $F_\Sigma \in \Diff^+(T^{d-1})$ supported in $\interior(\iota(D^{d-1}))$, and form the mapping torus $T^d\sharp\Sigma \coloneq ([0,1] \times T^{d-1})/((1,x)\sim(0,F_\Sigma(x))$. We parametrise this quotient by $[0,1)\times T^{d-1}$ in the evident way, use $[0,0]\in T^d \sharp \Sigma$ as base point, and we view $T^{d,\circ}$ as the complement of an embedded disc $D^d\subset T^d \sharp \Sigma$ that contains the part where the nontrivial gluing happened (i.e.\,the image of $\{0\}\times\iota(D^{d-1})$ in the quotient) and is disjoint from the image of $[0,1]\times \{0\}$ in the quotient. The latter is so that  the loop $([(t,0)])_{t\in [0,1]}$ in $T^d\sharp\Sigma$ is contained in $T^{d,\circ}$. We chose a basis for $\oH_1(T^{d,\circ})$ such that this loop represents $e_1\in \bfZ^d$.

The first claim we show is that the image of $e_1\in\bfZ^d$ under the map $\pp\colon  \pi_1\, T^d \sharp \Sigma\ra\pi_0\,\Diff^+_*(T^d \sharp \Sigma)$ is represented by the diffeomorphism $\phi \in \Diff^+_*(T^d \sharp \Sigma)$ given by using $\id\times f_\Sigma$ on the image of $\id \times \iota$ in $T^d\sharp\Sigma$, and extending it to all of $T^d\sharp\Sigma$ by the identity. This is because being the connecting map in the long exact sequence induced by the evaluation fibration $\ev_{[0,0]}\colon \Diff^+(T^d\sharp\Sigma)\ra T^d\sharp\Sigma$ with fibre $\Diff_\ast^+(T^d\sharp\Sigma)$, the point-pushing map sends $e_1\in \bfZ^d$ to the isotopy class of any diffeomorphism $\phi_1$ that arises as the value at time $t=1$ of a path $(\phi_t)_{t\in [0,1]}$ in $\Diff^+(T^d \sharp \Sigma)$ with $\phi_0 = \id$ and $\phi_t([0,0]) = [t,0]$. A possible choice of such path is given by $\phi_t([s,x])\coloneq [s+t,x]$ for $s+t< 1$ and $\phi_t([s,x])\coloneq [s+t-1,F_\Sigma(x)]$ for $s+t \geq 1$, which indeed agrees with $\phi$ at time $1$.

The second claim we make is that the image of $e_1\in\bfZ^d$ under the second composition in the statement is given by the diffeomorphism obtained by choosing an orientation-preserving embedding $\iota'\colon D^{d-1}\hookrightarrow T^{d-1}$ such that the image of $\id\times \iota'$ in $T^d\sharp\Sigma$ is contained in $T^{d,\circ}$ and avoids the origin, using $\id\times f_\Sigma$ on the image of $\id\times\iota'$ in $T^d\sharp\Sigma$ and extending it to a diffeomorphism of  $T^d\sharp\Sigma$ by the identity. This would imply the result, since the image of $e_1$ under both maps in consideration arises from the following construction: choose an embedding $S^1\times D^{d-1}\hookrightarrow T^d\sharp\Sigma\backslash\{[0,0]\}$ that represents $e_1\in\bfZ^d=\oH_1(T^d\sharp\Sigma)$ (which is unique up to isotopy as $d\ge4$), use $\id\times f_\Sigma$ on this image, and extend by the identity.

To show this claim, we prove more generally that the composition $\bfZ^d\otimes \Theta_d\le \Omega\le  \overline{\SL}_d(\bfZ)\ltimes \Omega\cong\pi_0\,\Diff_\partial(T^{d,\circ})$ is given by sending $x\otimes\Sigma'\in \bfZ^d\otimes \Theta_d$ to the diffeomorphism obtained by representing $x\in\pi_1\,T^{d,\circ}$ by an embedding $S^1\times D^{d-1}\subset T^{d,\circ}$ and $\Sigma'\in\Theta_d\cong\pi_0\,\Diff_\partial(D^d)$ by a diffeomorphism $f_{\Sigma'}\in \Diff_\partial(D^d)$, using $\id\times f_{\Sigma'}$ on  $S^1\times D^{d-1}$ and extending it to $T^d$ by the identity. By the argument from the proof of \cref{lem:plugging-in-disc}, it suffices to show that the described diffeomorphism considered as a diffeomorphism of $\pi_0\,\Diff_*(T^d)\cong\SL_d(\bfZ)\ltimes\Omega$ agrees with the image of $e_1\otimes \Sigma'$ under the inclusion $\bfZ^d\otimes \Theta_d\le \Omega$. This follows from \cite[p.~9, Remark (5)]{HatcherConcordance}.
\end{proof}

Combining \cref{lem:point-pushing} with \cref{lem:pointed-answer-complicated}, the long exact sequence induced by \eqref{equ:pp-fibration} implies:

\begin{cor}For $d\ge7$ and $\Sigma\in\Theta_d$ there is an isomorphism
	\[\pi_0\,\Diff^+(T^d\sharp\Sigma)\cong \big(\overline{\SL}_d(\bfZ)\ltimes \Omega \big)/\big(\langle (t_d,\eta\cdot\Sigma)\rangle \oplus (\bfZ^d \otimes \langle \Sigma \rangle) \big)\]
	which is compatible with the homomorphisms to $\SL_d(\bfZ)$.
\end{cor}

From this, the asserted identification of $\pi_0\,\Diff^+(T^d\sharp\Sigma)$ in \cref{bigthm:mcg-homotopy-tori} follows by proving that the right-hand quotient in the previous corollary can be simplified to the semidirect product
\[\begin{cases}\overline{\SL}_d(\bfZ)\ltimes\Big[
 \Omega/\big(\langle\eta\cdot\Sigma\rangle \oplus (\bfZ^d\otimes \langle\Sigma\rangle\big) \Big]&\text{if }\eta\cdot \Sigma\in\Theta_{d+1}\text{is not divisible by }2\\
\hfil \SL_d(\bfZ)\ltimes\Big[\Omega/ \big(\bfZ^d\otimes \langle\Sigma\rangle\big)\Big] &\text{if }\eta\cdot \Sigma\in\Theta_{d+1}\text{ is divisible by }2
\end{cases}\]
which follows by replacing $\Omega$ by $\Omega/(\bfZ^d\otimes \langle\Sigma\rangle)$ in the proof of \cref{lem:explicit-quotient}.

\renewcommand\thesubsubsection{\thesection.\arabic{subsection}}

\section{Splitting the homology action and the proof of \cref{bigthm:splitting}} To deduce \cref{bigthm:splitting} from \cref{bigthm:mcg-homotopy-tori}, we first determine for which homotopy tori $\cT$ the map $\pi_0\,\Diff^+(\cT)\ra \SL_d(\bfZ)$ is surjective. The following was stated in \cite[p.\,4]{BustamanteTshishiku} without proof.

\begin{lem}\label{lem:surjectivity}For a homotopy torus $\cT$ of dimension $d\neq4$, the map $\pi_0\,\Diff^+(\cT) \to \SL_d(\bfZ)$ is surjective if and only if $\cT$ is diffeomorphic to $T^d \# \Sigma$ for some $\Sigma\in\Theta_d$.
\end{lem}

\begin{proof}The direction $\Leftarrow$ is easy: if $\cT\cong T^d\sharp \Sigma$, then $\pi_0\,\Diff^+(\cT) \to \SL_d(\bfZ)$ is surjective because we can precompose it with the map $\ext_*\colon \pi_0\,\Diff_\partial((T^d)^\circ)\ra \pi_0\,\Diff^+(\cT)$ and use that $\pi_0\,\Diff_\partial((T^d)^\circ)\ra\SL_d(\bfZ)$ is surjective which holds for instance as a result of \cref{lem:mcg-torus-bdy}.

For the direction $\Rightarrow$, we may assume $d\ge5$ since for $d\le 3$ any torus $\cT$ is diffeomorphic to the standard torus $T^d$. This allows us to use smoothing theory \cite[Essay V]{KirbySiebenmann} which we briefly recall in a form suitable for our purposes: given a closed topological manifold $M$ of dimension $d\ge5$, the set $\Sm^{\con}(M)$ of concordance classes of smooth structures on $M$ is the set of equivalence classes of pairs $(T,\psi)$ of a smooth manifold $T$ together with a homeomorphism $\psi \colon T \to M$, where two pairs $(T,\psi)$ and $(T',\psi')$ are equivalent if there is a diffeomorphism $\Phi\colon T\ra T'$ such that the homeomorphisms $\psi$ and $\psi'\circ \Phi$ are concordant. The group $\smash{\pi_0\,\widetilde{\Homeo}(M)}$ of concordance classes of homeomorphisms acts on $\Sm^{\con}(M)$ by postcomposition and the set of orbits is in bijection with the set $\Sm^{\diff}(M)$ of diffeomorphism classes of smooth manifolds homeomorphic to $M$, induced by sending $(T,\psi)$ to $T$. There is a map
\[\eta\colon \mathrm{Sm}^{\con}(M) \lra \Lift(M,\BO\ra\BTOP)\]
to the set of isomorphism classes of pairs of a stable vector bundle over $M$ together with an isomorphism of the underlying stable Euclidean bundle with the stable topological tangent bundle $\tau^{\TOP}_M$ of $M$. The map $\eta$ is given by assigning a pair $(T,\psi)$ to the pullback $(\psi^{-1})^*\tau^{\Diff}_T$ of the stable tangent bundle of $T$ along $\psi^{-1}$, together with the isomorphism induced by the stable topological derivative of $\psi^{-1}$. The map $\eta$ turns out to be a bijection, by one of the main results of smoothing theory. Unwrapping definitions, one sees that the action of $\smash{\alpha\in \pi_0\,\widetilde{\Homeo}(M)}$ on $[T,\psi]\in\mathrm{Sm}^\con(M)$ is induced by pulling back the bundle along $\psi^{-1}$ and postcomposing the isomorphism with the stable topological derivative of $\psi^{-1}$. The set $\Lift(M,\BO\ra\BTOP)$ is a torsor for the group $[M,\TOP/\oO]$ of stable vector bundles on $M$ together with a trivialisation of the underlying stable Euclidean bundle; the group structure and the action are induced by taking direct sums. Thus,  if $M$ comes already equipped with a smooth structure then we obtain a bijection $\Lift(M,\BO\ra\BTOP)\cong [M,\TOP/\oO]$, postcomposition with which gives a bijection
\[\delta\colon \mathrm{Sm}^{\con}(M) \lra [M,\TOP/\oO].\]
Going through the definition, the action of $\smash{\alpha\in \pi_0\,\widetilde{\Homeo}(M)}$ on $[T,\psi]\in\mathrm{Sm}(M)$ translates to $\delta(T,\alpha\circ\psi)=(\alpha^{-1})^* \delta(T,\psi)+\delta(M,\alpha)$.

We now specialise to $M=T^d$. Given a homotopy torus $\cT$ of dimension $d\ge5$, a choice of homeomorphism $\varphi\colon \cT \to T^d$ induces a class $[\cT,\varphi]\in\Sm^{\con}(T^d)$ and a morphism $\pi_0\,\Diff^+(\cT)\ra \smash{\pi_0\,\widetilde{\Homeo}^+(T^d)}$ by conjugation with $\varphi$. This agrees with the map to $\SL_d(\bfZ)$ when precomposed with the action map $\smash{\pi_0\,\widetilde{\Homeo}^+(T^d)}\ra\SL_d(\bfZ)$. The latter is an isomorphism as a result of  \cref{lem:pointed-block-homeo} and the isomorphism $\smash{\pi_0\,\widetilde{\Homeo}^+_\ast(T^d)\cong \pi_0\,\widetilde{\Homeo}^+(T^d)}$ (use the action of $T^d$ on itself to see this), so it suffices to show that $\pi_0\,\Diff^+(\cT)\ra \smash{\pi_0\,\widetilde{\Homeo}^+(T^d)}$ is not surjective unless $\cT$ is diffeomorphic to $T^d\sharp\Sigma$ for some $\Sigma\in\Theta_d$. The image of $\pi_0\,\Diff^+(\cT)\ra \smash{\pi_0\,\widetilde{\Homeo}(T^d)}$ is contained in the stabiliser of $[\cT,\varphi]\in\Sm^{\con}(T^d)$, so it is enough to show that $[\cT,\varphi]\in \Sm(T^d)$ is not contained in the invariants of this action unless $\cT\cong T^d\sharp\Sigma$ for some $\Sigma\in\Theta_d$. Since any $\smash{\alpha\in \pi_0\,\widetilde{\Homeo}^+(T^d)\cong\SL_d(\bfZ)}$ is isotopic to a diffeomorphism of $T^d$, the terms $\delta(T^d,\alpha)$ in the above description of the $\pi_0\,\widetilde{\Homeo}(T^d)$-action $\mathrm{Sm}^{\con}(T^d) \cong [T^d,\TOP/\OO]$ vanishes, and thus the action is simply by precomposition. In particular, it is an action by group homomorphisms if we equip $[T^d,\TOP/\OO]$ with the group structure induced by the infinite loop space structure on $\TOP/\OO$. Using this infinite loop space structure and the fact that $T^d$ stably splits into a wedge of spheres we also get a direct sum decomposition of $\SL_d(\bfZ)$-modules $\smash{[T^d,\TOP/\OO] \cong \oplus_{r=1}^d \Hom(\Lambda^r \bfZ^d,\pi_r\,\TOP/\OO)}$. We will show below that the invariants of this action are given by the subgroup $\Hom(\Lambda^d \bfZ^d,\pi_d\,\TOP/\OO)\cong\pi_d\,\TOP/\OO\cong\Theta_d$. This will imply the claim, since the subgroup $\Theta_d\le [T^d,\TOP/\OO]\cong\Sm^{\con}(T^d)$ corresponds to the classes of the pairs $(T^d\sharp \Sigma,\id_{T^d}\sharp \beta)$ where $\beta\colon \Sigma\ra S^d$ is the unique homeomorphism up to isotopy that fixes the disc where the connected sum is taken, so in particular  $[\cT,\varphi]\in\Sm^{\con}(T^d)$ is not contained in this subgroup unless $\cT$ is diffeomorphic to $T^d\sharp \Sigma$ for some $\Sigma\in\Theta_d$.

To finish the proof, it thus suffices to show that for a finitely generated abelian group $A$, the $\SL_d(\bfZ)$-action on $\Hom(\Lambda^r \bfZ^d,A)$ by precomposition with the inverse has no invariants for $0<r<d$. This is isomorphic to the standard action on $\Lambda^r \bfZ^d\otimes A$ up to the automorphism of $\SL_d(\bfZ)$ given by taking inverse transpose, so we may equivalently show that $\Lambda^r \bfZ^d\otimes A$ has no invariants . Without loss of generality we may assume that $A=\bfZ/n$ is cyclic. In this case, $\Lambda^r \bfZ^d\otimes \bfZ/n$ has a basis as a $\bfZ/n$-module indexed by subsets $I\subset \{1,\ldots,d\}$ of cardinality $r$, where the basis vector $x_I$ corresponding to $I\subset \{1,\ldots,d\}$  is $x_{i_1} \wedge \cdots \wedge x_{i_r}$ for $i_1<\ldots<i_r$ and $I=\{i_1,\ldots,i_r\}$ where $x_1,\ldots x_d$ is the standard $\bfZ/n$-basis of  $\bfZ^d\otimes \bfZ/n$. Now observe that an elementary matrix $(I+E_{ij})\in\SL_d(\bfZ)$ for $1\le i,j\le d$ acts by sending $x_{I}$ to $x_{I} \pm x_{(I\backslash i)\cup j}$ if $i \in I$ and $j \notin I$, and to itself otherwise. On a general element $\smash{v=\sum_I \lambda_I(v)\cdot x_I \in  \Lambda^r \bfZ^d\otimes \bfZ/n}$, the matrix $(I+E_{ij})$ thus acts by
 \[\textstyle{v\longmapsto\underset{i\in I\text{ or }j\not\in I}{\sum}\lambda_I(v)\cdot x_I + \underset{i\not\in I\text{ and }j\in I}{\sum}(\lambda_I(v)+\lambda_{(I\backslash j)\cup i}(v))\cdot x_I},\]
so if $v$ is an invariant, then $\lambda_{(I\backslash j)\cup i}(v)=0$ for all $I$ with $i\not\in I$ and $j\in I$. But since $1\le i,j\le d$ were arbitrary and $0<r<d$, these are in fact all coefficients, so all invariants are zero.
\end{proof}

\subsection{Proof of \cref{bigthm:splitting}} We conclude this section with the proof of \cref{bigthm:splitting}, which says that the map $\pi_0\,\Diff^+(\cT) \to \SL_d(\bfZ)$ given by the action on $H_1(\cT) = \bfZ^d$ admits a splitting if and only if $\cT = T^d \sharp \Sigma$ for $\Sigma\in\Theta_d$ such that $\eta \cdot \Sigma \in \Theta_{d+1}$ divisible by 2.

\begin{proof}[Proof of \cref{bigthm:splitting}] We distinguish the cases whether a given homotopy torus $\cT$ of dimension $d\neq4$ is diffeomorphic to $T^d\sharp\Sigma$ for some $\Sigma\in\Theta_d$ or not. If it is not, then the map $\pi_0\,\Diff^+(\cT)\ra\SL_d(\bfZ)$ is not surjective by \cref{lem:surjectivity}, so it is in particular not split surjective. If it is, then by \cref{bigthm:mcg-homotopy-tori} the group $\pi_0\,\Diff^+(\cT)$ is isomorphic, compatibly with the map to $\SL_d(\bfZ)$, to a semidirect product of $\SL_d(\bfZ)$ or $\overline{\SL}_d(\bfZ)$ depending on whether $\eta\cdot \Sigma \in \Theta_{d+1}$ is divisible by $2$ or not. In the first case, the map to $\SL_d(\bfZ)$ visibly admits a splitting. In the second case, a hypothetical splitting would in particular induce a splitting of the projection $\overline{\SL}_d(\bfZ)\ra \SL_d(\bfZ)$, which does not exist since this extension is nontrivial. This finishes the proof.
\end{proof}

\section{Endomorphisms of $\SL_d(\bfZ)$ and the proofs of Theorems~\ref{thm:action-implies-splitting} and \ref{thm:SLdZ}}
\label{sec:SLdZ}

This section serves to deduce \cref{thm:action-implies-splitting} from the classification result for endomorphisms of $\SL_d(\bfZ)$ stated as \cref{thm:SLdZ}, and to prove the latter.

\subsection{Proof of \cref{thm:action-implies-splitting} assuming \cref{thm:SLdZ}}
Assuming \cref{thm:SLdZ}, we prove \cref{thm:action-implies-splitting}. We first assume $G=\Homeo^+(\cT)\cong \Homeo^+(T^d)$. Given a nontrivial homomorphism $\varphi\colon \SL_d(\bfZ)\ra \Homeo^+(T^d)$ for $d\ge3$, the composition with the action on homology $\Homeo^+(T^d)\ra \SL_d(\bfZ)$ is by \cref{thm:SLdZ} either trivial or an isomorphism, so we have to exclude the former. If it were trivial, then $\varphi$ would have image in $\Tor^{\TOP}(T^d) = \ker(\Homeo^+(T^d)\ra \SL_d(\bfZ))$. Suppose for contradiction that $\varphi\colon \SL_d(\bfZ)\ra \Tor^{\TOP}(T^d)$ is nontrivial. Its kernel is a normal subgroup, so by \cite[Corollary 1, p.~36]{Mennicke} it is either (a) contained in the centre $Z(\SL_d(\bfZ))$, which is trivial or $\bfZ/2$ depending on the parity of $d$, or (b) of finite index. In either case, the image of $\phi$ contains a nonabelian finite group $H$: in case (a) it contains $\SL_d(\bfZ)$ or $\PSL_d(\bfZ)$, so in particular a nonabelian finite group $H$, and in case (b) the image of $\varphi$ is finite itself, and also nonabelian since otherwise $\varphi$ would be trivial since $\SL_d(\bfZ)$ is perfect for $d \geq 3$ (see \cref{lem:low-degree-homology}).

To make use of the nonabelian finite subgroup $H\le \Tor^{\TOP}(T^d)$, following \cite{LeeRaymond}, we consider the extension $0\ra \bfZ^d\ra N_{\Homeo(\smash{\widetilde{T^d}})}(\bfZ^d)\ra \Homeo(T^d)\ra 0$ whose middle group is the normalizer of $\bfZ^d=\pi_1(T^d)$ considered as a subgroup of the homeomorphism group $\Homeo(\smash{\widetilde{T^d}})$ of the universal cover. Note that the induced action of $\Homeo(T^d)$ by $\bfZ^d$ agrees by construction with the action on the fundamental group. The pullback $0\ra \bfZ^d\ra E\ra H\ra 0$ of this extension along $H\le \Homeo(T^d)$ is, by the Corollary on p.~256 of loc.cit.\,admissible in the sense of p.~256 loc.cit.. The proof of Proposition 2 loc.cit.\,then shows that the centraliser $C_E(\bfZ^d)$ of $\bfZ^d$ in $E$ is abelian. But since $H\le \Tor^{\TOP}(T^d)$ acts trivially on $\bfZ^d$ we have $C_E(\bfZ^d)=E$, so $E$ is abelian and thus the same holds for $H$ which cannot be true by the choice of $H$, so $\varphi$ has to be trivial.

The case $G=\Diff^+(\cT)$ follows from the case $\Homeo^+(\cT)$ by postcomposing a given homomorphism $\varphi\colon \SL_d(\bfZ)\ra \Diff^+(\cT)$ with the inclusion $\Diff^+(\cT)\le \Homeo^+(\cT)$, so we are left to prove the addendum concerning homomorphisms from $\SL_d(\bfZ)$ into $\pi_0\,\Diff^+(\cT)$ or $\pi_0\,\Homeo^+(\cT)$ under the additional assumption $d\neq4,5$. By the same argument as before, it suffices to show that all homomorphisms from $\SL_d(\bfZ)$ into $\pi_0\,\Tor^{\Diff}(\cT)$ or $\pi_0\,\Tor^{\TOP}(\cT)$ are trivial. \cref{lem:torelli-abelian} below says that the latter two groups are abelian, so such morphisms factor over the (trivial) abelianisation of $\SL_d(\bfZ)$ (see \cref{lem:low-degree-homology}) and are therefore trivial, as claimed.

\begin{lem}\label{lem:torelli-abelian}For a homotopy torus $\cT$ of dimension $d\neq 4,5$, the kernels of the homology actions
\begin{align*}\pi_0\,\Tor^{\TOP}(\cT)&=\ker\big[\pi_0\,\Homeo^{+}(\cT)\ra \SL_d(\bfZ)\big]\\
\pi_0\,\Tor^{\Diff}(\cT)&=\ker\big[\pi_0\,\Diff^{+}(\cT)\ra \SL_d(\bfZ)\big]\end{align*}
are both abelian.
\end{lem}
\begin{proof}For $d\le 3$, the homotopy torus $\cT$ is diffeomorphic to the standard torus and both kernels $\pi_0\,\Tor^{\TOP}(\cT)$ and $\pi_0\,\Tor^{\Diff}(\cT)$ are trivial, so in particular abelian. To show the claim for $d\ge6$, note that $\pi_0\,\Tor^{\TOP}(\cT)\cong \pi_0\,\Tor^{\TOP}(T^d)$ because $\cT$ is homeomorphic to $T^d$, so $\pi_0\,\Tor^{\TOP}(\cT)$ is abelian since we have $\pi_0\,\Tor^{\TOP}(T^d)\cong (\bfZ/2)^\infty$ by \cite[Theorem 4.1]{HatcherConcordance}. To show that $\pi_0\,\Tor^{\Diff}(\cT)$ is abelian, note that as $\cT\simeq K(\bfZ^d,1)$ we may view the map $\pi_0\,\Diff^{+}(\cT)\ra \SL_d(\bfZ)$ as the induced map on path components of the map $\Diff^{+}(\cT)\ra \hAut^+(\cT)$ to the space of orientation-preserving homotopy equivalences, so $\pi_0\,\Tor^{\Diff}(\cT)$ receives an epimorphism from $\pi_1(\hAut^+(\cT)/\Diff^{+}(\cT))$. Replacing $T^d$ by $\cT$ in the argument for (3) on page 8 of \cite{HatcherConcordance} and using that $\cT$ is homeomorphic to $T^d$, we get that  $\pi_1(\hAut^+(\cT)/\Diff^{+}(\cT))$ is isomorphic to the abelian group $(\oplus_{0 \leq j \leq d} (\Lambda^{j} \bfZ^d)\otimes \Theta_{d-j+1})\oplus ( (\Lambda^{d-2} \bfZ^d)\otimes \bfZ/2)\oplus   \bfZ/2^{\infty}$, so the claim follows (the final step can also be proved via smoothing theory).
\end{proof}

\subsection{Proof of \cref{thm:SLdZ}} In the remainder of this section, we prove \cref{thm:SLdZ}. The proof makes use of the subgroup $\bfU_d<\SL_d(\bfZ)$ of unipotent upper triangular matrices which in particular contains the elementary matrices $E_{ij}$ for $1 \leq i<j \leq d$; these have $1$ on the diagonal and at the $(i,j)$th entry, and $0$ at all other entries. It is well-known that $\bfU_d$ is an $(d-1)$-step nilpotent group whose centre is generated by the elementary matrix $E_{1d}$, which is an iterated commutator of length $(d-1)$, namely $E_{1d}=[E_{12},[E_{23},[\ldots,[E_{d-2,d-1},E_{d-1,d}]]]]$. An important ingredient in the proof of \cref{thm:SLdZ} is the following lemma on complex representations of $\bfU_d$. In its statement and in all that follows, we write \[(-)^{-t}\colon \SL_d(\bfZ)\lra \SL_d(\bfZ)\] for the automorphism of $\SL_d(\bfZ)$ given by taking inverse-transpose.

\begin{lem}\label{lem:unipotent}
Fix $d\ge3$ and a homomorphism $\phi \colon \bfU_d\ra \GL_m(\bfC)$ with $m\le d$.
\begin{enumerate}
\item \label{enum:unipotent-i} Assume $d\ge4$. If $m<d$, or if $m=d$ and $\phi(E_{1d})$ is not a scalar, then $\phi(E_{1d})$ is unipotent.
\item \label{enum:unipotent-ii} If $m<d$ and $\phi(E_{ij})$ is unipotent for each $i<j$, then $\phi(E_{1d})=\id$. 
\item \label{enum:unipotent-iii} If $m=d$ and $\phi(E_{ij})$ is unipotent for each $i<j$ and $\phi(E_{1d})\neq\id$, then $\phi(E_{1d})-\id$ has rank $1$.
\item \label{enum:unipotent-iv} If $\phi(E_{ij})$ is unipotent and $\phi(E_{ij})-\id$ has rank $1$ for all $i<j$, then after possibly precomposing $\phi$ with $(-)^{-t}$, the matrices $\phi(E_{1d}),\ldots, \phi(E_{d-1d})$ all have the same fixed set.
\end{enumerate}
\end{lem}

\begin{rem}\label{rmk:bader}
The argument in Bader's MathOverflow post \cite{mathoverflow} contains the claim that for any representation $\phi\colon \bfU_3\ra\GL_3(\bfC)$ the matrix $\phi(E_{13})$ is unipotent. This is incorrect:  \cref{lem:dim3} \ref{enum:dim3d} gives a representation $\bfU_3\ra\GL_3(\bfC)$ for which $\phi(E_{13})$ is a nontrivial scalar. Also \cref{lem:unipotent} \ref{enum:unipotent-i} fails for $d=m=3$ (the case \ref{enum:dim3c} of \cref{lem:dim3} below involves representations $\bfU_3\ra\GL_3(\bfC)$ for which $\phi(E_{13})$ is a non-scalar semisimple matrix).
\end{rem}

 For $d=3$, we will circumvent the use of \cref{lem:unipotent} \ref{enum:unipotent-i} in later proofs by means of the following:

\begin{lem}\label{lem:dim3}
Fix a homomorphism $\phi\colon\bfU_3\ra\GL_m(\bfC)$. If $m=2$, then either
\begin{enumerate}
\item \label{enum:dim3firsti} $\phi(E_{13})$ is unipotent, or
\item \label{enum:dim3firstii} there are $\mu,\nu\in\bfC^\times$ and $C\in\GL_2(\bfC)$ so that after postcomposing $\phi$ with conjugation by $C$,
\[E_{12}\mapsto \left(\begin{array}{ccc}\mu&0\\0&-\mu
\end{array}\right)\>\>\>\>\>\>
E_{23}\mapsto\left(\begin{array}{ccc}0&\nu\\1&0\\
\end{array}\right)\>\>\>\>\>\>
E_{13}\mapsto\left(\begin{array}{ccc}-1&0\\0&-1\\
\end{array}\right).
\]
\end{enumerate}
 If $m=3$, then either
\begin{enumerate}
\item \label{enum:dim3b} $\phi(E_{13})$ is unipotent,
\item \label{enum:dim3c} $\phi=\phi_1\oplus\phi_2$, where $\phi_i\colon\bfU_3\ra\GL_i(\bfC)$, up to conjugation, or
\item \label{enum:dim3d} there are $\lambda,\mu,\nu\in\bfC^\times$ and $C\in\GL_3(\bfC)$ so that $\lambda$ is a nontrivial cube root of $1$, and after postcomposing $\phi$ with conjugation by $C$,
\[E_{12}\mapsto\left(\begin{array}{ccc}\mu&0&0\\0&\lambda\mu&0\\0&0&\lambda^2\mu
\end{array}\right)\>\>\>\>\>\>
E_{23}\mapsto\left(\begin{array}{ccc}0&0&\nu\\1&0&0\\0&1&0
\end{array}\right)\>\>\>\>\>\>
E_{13}\mapsto\left(\begin{array}{ccc}\lambda&0&0\\0&\lambda&0\\0&0&\lambda
\end{array}\right).
\]
\end{enumerate}
\end{lem}

We omit the proof of \cref{lem:dim3} since it is based on similar (and easier) analysis as the base case in the proof of \cref{lem:unipotent} \ref{enum:unipotent-i} which we explain now.

\begin{proof}[Proof of Lemma \ref{lem:unipotent} \ref{enum:unipotent-i}] We do an induction on $d$. To simplify the notation we set $u_{ij}\coloneq\phi(E_{ij})$.

\smallskip

\noindent {\it Base case.} We treat the case $d=4$ by hand. To show that $u_{14}$ is unipotent, it suffices to prove that all its eigenvalues $\lambda$ equal $1$. Let $V_{\lambda}$ be the $\lambda$-eigenspace for $u_{14}$. Since $E_{14}$ is central in $\bfU_4$, restricting to $V_\lambda$ gives a homomorphism $\bfU_4\ra \GL(V_\lambda)$ whose image of $E_{ij}$  we denote by $u_{ij}'$. Next we distinguish cases depending on the dimension of $V_\lambda$.
By the assumption that $u_{14}=\phi(E_{14})$ is not a scalar when $m=d$, we know $\dim (V_\lambda)\le 3$. If $\dim V_\lambda=1$, then since $\GL_1(\bfC)=\bfC^\times$ is abelian, we have $u_{14}'=1$ because $E_{14}=[E_{13},E_{34}]$ is a commutator, so $\lambda=1$. If $\dim V_\lambda=2$, we consider the subgroup $\langle u_{13}',u_{34}',u_{14}'\rangle\le \GL(V_\lambda)$ generated by the images of $E_{13}$, $E_{34}$, and $E_{14}$ in $\GL(V_\lambda)$. By assumption $u_{14}'=\lambda\cdot \id_{2\times 2}$. Let $x \in V_\lambda$ be an eigenvector for $u_{13}'$ with eigenvalue $\mu$. Using the relation $[E_{13},E_{34}]=E_{14}$ we conclude that $(u_{34}')^i(x)$ is an eigenvector for $u_{13}'$ with eigenvalue $\lambda^i\mu$. Since $\dim V_\lambda=2$, this forces $\lambda^2=1$ because eigenvectors with different eigenvalues are linearly independent, and thus $\lambda=\pm1$. Suppose for a contradiction that $\lambda=-1$. Then $u_{13}'$ has two distinct eigenvalues $\mu$ and $-\mu$. Since $E_{13}$ is central in $\langle E_{12},E_{23}\rangle\cong\bfU_3$, we deduce that $u_{12}'$ and $u_{23}'$ are simultaneously diagonalisable; in particular they commute. But since $E_{13}=[E_{12},E_{23}]$, this implies $u_{13}'=\id$, which is a contradiction, so $\lambda$ has to be $1$. Finally, suppose that $\dim V_\lambda=3$. In this case the argument is very similar to the preceding case: by assumption $u_{14}'=\lambda\cdot \id_{3\times 3}$, and the relation $[E_{13},E_{34}]=E_{14}$ implies that $\langle\lambda\rangle\subset\bfC^\times$ acts freely on the eigenvalues of $u_{13}'$ which implies $\lambda^3=1$. If $\lambda\neq1$, then $u_{13}'$ has distinct eigenvalues $\mu$, $\lambda\mu$, and $\lambda^2\mu$ for some $\mu$. Using the fact that $E_{13}$ is both central and a commutator in $\langle E_{12},E_{23}\rangle\subset\bfU_4$, we reach a contradiction.

\smallskip

\noindent {\it Induction step.} Fix an eigenvalue $\lambda$ for $u_{1d}=\phi(E_{1d})$, and let $V_\lambda$ be the corresponding eigenspace. We have to show $\lambda=1$. As in the base case, since $u_{1d}$ is not a scalar if $d=m$, so $\dim(V_\lambda)\le d-1$ and since $E_{1d}$ is central in $\bfU_d$, the representation restricts to $\bfU_d\ra\GL(V_\lambda)$. As before we write $\smash{u_{ij}'}$ for the image of $E_{ij}$ under this homomorphism. Consider the subgroup $\langle E_{1,d-1},E_{d-1,d},E_{1,d}\rangle\cong\bfU_3$ of $\bfU_d$. Let $\mu$ be an eigenvalue of $\smash{u_{1,d-1}'}$ and let $V_\mu\le V_\lambda$ be the corresponding eigenspace. As above, $\lambda^i\mu$ is also an eigenvalue for $\smash{u_{1,d-1}'}$ for each $i$. Consider the subgroup $\langle E_{ij} \mid i<j\le d-1\rangle\cong \bfU_{d-1}$ of $\bfU_d$. Since $E_{1,d-1}$ is central in this copy of $\bfU_{d-1}$, there is an induced map $\bfU_{d-1}\ra\GL(V_\mu)$, $E_{ij}\mapsto u_{ij}''$. If $\lambda\neq1$, then $V_\mu$ is a proper subspace of $V_\lambda$ (since the eigenspaces for $\mu$ and $\lambda\mu$ are linearly independent). Then $\dim(V_\mu)\le d-2$, so the induction hypothesis implies that $\smash{u_{1,d-1}''}$ is unipotent, so $\mu=1$.  Since the same argument applies for each eigenspace of $u_{1,d-1}'$, we conclude that $1=\mu=\lambda\mu$, so $\lambda=1$ as claimed.
\end{proof}

\begin{proof}[Proof of Lemma \ref{lem:unipotent} \ref{enum:unipotent-ii}]
Fixing $\phi:\bfU_d\ra\GL_{m}(\bfC)$ such that $\phi(E_{ij})$ is unipotent for all $1 \leq i<j \leq d$, we want to show $\phi(E_{1d})=\id$. As before we write $u_{ij}=\phi(E_{ij})$. Note that the special case $m=d-1$ implies the case $m<d-1$, because if $m<d-1$ then we may restrict $\phi$ to the subgroup $\bfU_{m+1}\cong \langle E_{ij} \mid 1\le i<j\le m+1\rangle\le \bfU_d$ to conclude $u_{1,m+1}=\id$ from the special case, so using $E_{1d}=[E_{1,m+1},E_{m+1,d}]$ we get $u_{1d}=[u_{1,m+1},u_{m+1,d}]=\id$. To prove the special case $m=d-1$, we do an induction on the dimension $d$. 

\smallskip

\noindent {\it Base case.} To settle the case $\phi \colon \bfU_3\ra\GL_2(\bfC)$, suppose for a contradiction that $u_{13}$ is not the identity. Since it is unipotent by assumption, it has up to conjugation the form $\left(\begin{smallmatrix}
1&1\\0&1\end{smallmatrix}\right)$, so by postcomposing $\phi$ with this conjugation we may assume that $u_{13}$ equals this matrix. Since $E_{13}$ is central in $\bfU_3$, the image of $\phi$ is contained in the centraliser of $\left(\begin{smallmatrix}
1&1\\0&1\end{smallmatrix}\right)$ which consists of matrices of the form $\left(\begin{smallmatrix}
a&b\\0&a\end{smallmatrix}\right)$. This is an abelian subgroup,  so $u_{13}=[u_{12},u_{23}]$ is identity, a contradiction. 

\smallskip

\noindent {\it Induction step.} For the induction step, we fix $\phi \colon \bfU_d\ra\GL_{d-1}(\bfC)$ and suppose for a contradiction that $u_{1d}\neq\id$. Consider the subspaces $K_1\subset K_2\subset\bfC^{d-1}$ where $K_i=\ker(u_{1d}-\id)^i$. Writing $k_i\coloneq \dim K_i$ we have $k_1>0$ since $u_{1d}$ is unipotent, $\ell \coloneqq k_2-k_1>0$ since $u_{1d}\neq\id$, and $\ell\le k_1$ (one way to see this is to consider the Jordan normal form). Note that since $E_{1d}$ is central in $\bfU_d$, the image of $\phi$ preserves $K_2$ so we obtain a morphism $\phi'\colon \bfU_d\ra \GL(K_2)$ by restriction. We write $u_{ij}'$ for its image of $E_{ij}$. Setting $m\coloneq k_1-\ell\ge0$, we choose a basis for $K_2$ that extends a basis for $K_1$ and that has the property that
\begin{equation}\label{eqn:E1d-matrix}u_{1d}'=\left(
\begin{array}{ccc}
\id_{\ell\times\ell}&0&\id_{\ell\times\ell}\\
0&\id_{m\times m}&0\\
0&0&\id_{\ell\times\ell}.
\end{array}\right)\end{equation}
in this basis. To see that such a basis exists, it is again helpful to use the Jordan normal form. Since $E_{1d}$ is central in $\bfU_d$, the morphism $\bfU_d\ra \GL(K_2)\cong\GL_{2\ell+m}(\bfC)$ lands in the centraliser of \eqref{eqn:E1d-matrix} which are the matrices of the form
\begin{equation}\label{eqn:centralizer}\left(
\begin{array}{ccc}
A&X&Z\\
0&B&Y\\
0&0&A
\end{array}\right).\end{equation}
We claim that $\smash{u'_{1,\ell+m+1}}$ and $\smash{u'_{\ell+m+1,2\ell+m+1}}$ have the form
\begin{equation}\label{eqn:E1j-matrix}u'_{1,\ell+m+1}=\left(
\begin{array}{ccc}
\id_{\ell\times \ell}&0&Z\\
0&\id_{m\times m}&0\\
0&0&\id_{\ell\times\ell} 
\end{array}\right)\quad \text{and}\quad u'_{\ell+m+1,2\ell+m+1}=\left(
\begin{array}{ccc}
\id_{\ell\times\ell}&X&Z'\\
0&B&Y\\
0&0&\id_{\ell\times\ell}
\end{array}\right)
\end{equation}
for some $B,X,Y,Z$, and $Z'$. Assuming this claim for now, we observe that the matrices \eqref{eqn:E1j-matrix} commute, so $\smash{u'_{1,2\ell+m+1}=[u'_{1,\ell+m+1},u'_{\ell+m+1,2\ell+m+1}]}$ is the identity. If $2\ell+m+1=d$ then we are done since this contradicts \eqref{eqn:E1d-matrix}. If $2\ell+m+1<d$ then the relation  $\smash{u_{1d}'=[u'_{1,2\ell+m+1},u'_{2\ell+m+1,d}]}$ shows that $u_{1d}'$ is the identity, which again contradicts \eqref{eqn:E1d-matrix}. This leaves us with showing \eqref{eqn:E1j-matrix}. We first treat $\smash{u'_{1,\ell+m+1}}$. Since $\phi'$ has image in \eqref{eqn:centralizer}, we may postcompose it with 
\[\left(
\begin{array}{ccc}
A&X&Z\\
0&B&Y\\
0&0&A
\end{array}\right)
\mapsto\left(\begin{array}{cc}A&X\\0&B\end{array}\right)\>\>\>\text{ and }\>\>\>
\left(
\begin{array}{ccc}
A&X&Z\\
0&B&Y\\
0&0&A
\end{array}\right)
\mapsto\left(\begin{array}{cc}B&Y\\0&A\end{array}\right).
\]
to obtain two homomorphisms $\bfU_d\ra \GL_{\ell+m}(\bfC)$. We may apply the induction hypothesis to the restriction of these to the subgroup $\bfU_{\ell+m+1}\cong \langle E_{1,i}\mid 1<i\le \ell+m+1\rangle$ to conclude that the image of $\smash{u'_{1,\ell+m+1}}$ under these two homomorphism is the identity, so $\smash{u'_{1,\ell+m+1}}$ has the claimed form.

To deal with the second matrix $\smash{u'_{\ell+m+1,2\ell+m+1}}$ we argue similarly: postcompose $\phi'$ with the restriction to $A$ to obtain a morphism $\bfU_d\ra \GL_{\ell}(\bfC)$, restrict them to the subgroup $\bfU_{\ell+1}\cong \langle E_{\ell+m+1,i}\mid \ell+m+1<i\le 2\ell+m+1\rangle$ in $\bfU_d$, and apply the induction hypothesis.
\end{proof}

\begin{proof}[Proof of \cref{lem:unipotent} \ref{enum:unipotent-iii}] Fix $\phi \colon \bfU_d\ra\GL_d(\bfC)$ such that $u_{1d}=\phi(E_{1d})$ is unipotent and nontrivial. The subspace $K_2=\ker(u_{1d}-\id)^2$ is nontrivial, preserved by the image of $\phi$,  each $\phi(E_{ij})$ acts on it by a nontrivial unipotent, and $\phi(E_{1d})$ acts nontrivially on it, so \cref{lem:unipotent} \ref{enum:unipotent-ii} implies $K_2=\bfC^d$. Arguing as in the proof of \cref{lem:unipotent} \ref{enum:unipotent-ii}, up to changing basis (corresponding to postcomposing $\phi$ with a conjugation), we can assume that \eqref{eqn:E1d-matrix} holds and by the same argument as in the previous proof $\phi$ has image in matrices of the form \eqref{eqn:centralizer} and $u_{1,\ell+m+1}$ has the form \eqref{eqn:E1j-matrix}. We are left to show $\ell=1$ since then $u_{1d}$ has rank $1$ in view of \eqref{eqn:E1d-matrix}. Assuming for a contradiction that $\ell>1$, then $\ell+m+1 < d$, so $u_{1d}=[u_{1,\ell+m+1},u_{\ell+m+1,d}]$. Written out in matrices this equation reads as
\[
\left(
\begin{array}{ccc}
\id&0&\id\\
0&\id&0\\
0&0&\id
\end{array}\right)=\left(\begin{array}{ccc}
\id&0&Z\\
0&\id&0\\
0&0&\id
\end{array}\right)
\left(\begin{array}{ccc}
A&X&Z'\\
0&B&Y\\
0&0&A
\end{array}\right)\left(\begin{array}{ccc}
\id&0&-Z\\
0&\id&0\\
0&0&\id
\end{array}\right)\left(\begin{array}{ccc}
A&X&Z'\\
0&B&Y\\
0&0&A
\end{array}\right)^{-1}\]
which implies $ZA-AZ=A$, but this is a contradiction because the trace of $ZA-AZ$ is $0$, whereas that of $A$ is nonzero since $A$ is unipotent because so is $u_{1d}'$, by assumption. 
\end{proof}

Before proving Lemma \ref{lem:unipotent} \ref{enum:unipotent-iv}, we discuss some properties of rank-1 operators. Given subspaces $H,L\le \bfC^d$ with $\dim H=d-1$, $\dim L=1$, there is a rank-1 operator with kernel $H$ and image $L$, which is unique up to a unit, namely the composition
$\bfC^d\onto \bfC^d/H\cong\bfC\ra\bfC\cong L\hra\bfC^d$. In what follows, it will be convenient to consider rank-1 operators up to scalars; abusing notation, we will use $\Pi_{H,L}$ to denote either this equivalence class of rank-1 operator with kernel $H$ and image $L$. In terms of equivalence classes, the composition behaves as
\[\Pi_{H,L}\circ\Pi_{H',L'}=\begin{cases}
0&\text{ if }L'\subset H,\\
\Pi_{H',L}&\text{ if }L'\not\subset H.
\end{cases}
\]
The operator $u_{H,L} \coloneqq \id+\Pi_{H,L}$ (which is well-defined up to scaling $u_{H,L}-\id$ by a unit) is unipotent if and only if $L\subset H$ (otherwise $u_{H,L}$ is diagonalisable and nontrivial). In this case the fixed set of $u_{H,L}$ is $H$ and its inverse is $\id-\Pi_{H,L}$ which is another representative of $u_{H,L}$. 
Fixing two such equivalence classes of unipotent operators $u_{H,L}$ and $u_{H',L'}$, we have the commutator relation 
\begin{equation}\label{eqn:rank1}[u_{H,L}\>,\>u_{H',L'}]=
\begin{cases}
u_{H',L}&\text{ if }L\subset H'\text{ and }L'\not\subset H,\\
u_{H,L'}&\text{ if }L\not\subset H'\,\,\text{and }L'\subset H,\\
\id&\text{ if }L\subset H'\ \,\text{ and }L'\subset H.
\end{cases}
\end{equation}
If $L\not\subset H'$ and $L'\not\subset H$, then the commutator is not unipotent. 

The following observation will play a role in the proof of \cref{lem:unipotent} \ref{enum:unipotent-iv}: Fixing unipotent operators $u_{H_i,L_i}$ as above for $i=1,2,3$ and assuming firstly that $u_{H_1,L_1}$ commutes with $u_{H_j,L_j}$ for $j=2,3$ and secondly that
$u_{H_1,L_1}=[u_{H_2,L_2}\>,\>u_{H_3,L_3}]$, we may use the commutator formula from above to conclude that $L_j\subset H_1$ for $j=2,3$ and that $H_2=H_1$ or $H_3=H_1$.

\begin{proof}[Proof of Lemma \ref{lem:unipotent} \ref{enum:unipotent-iv}]
Since $\phi(E_{ij})-\id$ has rank $1$ for $i<j$, the operators $u_{ij}\coloneq \phi(E_{ij})=\id + (\phi(E_{ij})-\id)$ are for $i<j$ of the form $u_{H_{ij},L_{ij}}$ as discussed above where $H_{ij}$ is the kernel of $\phi(E_{ij})-\id$, i.e.\,the fixed set of $\phi(E_{ij})$. We claim that either  $H_{1d}=H_{2,d}=\cdots=H_{d-1,d}$ or $H_{1d}=H_{1,d-1}=\cdots=H_{12}$. This would imply the result, because the two cases are interchanged when precomposing $\phi$ with $(-)^{-t}$. To show this claim, we use that $u_{1d}$ commutes with $u_{ij}$ for $i<j$. Since $u_{1d}=[u_{12},u_{2d}]$, it follows from the discussion after \eqref{eqn:rank1} that either $H_{12}=H_{1d}$ or $H_{2d}=H_{1d}$. In the first case, we also have $H_{1j}=H_{1d}$ for all $2\le j\le d$, using $u_{1j}=[u_{12},u_{2j}]$ and the fact that $u_{2j}$ preserves $H_{1d}$ since it commutes with $u_{1d}$. Similarly, in the second case we also have $H_{j,d}=H_{1d}$ for all $2\le j\le d$ using $u_{j,d}=[u_{j,2},u_{2,d}]$ and that $u_{j,2}$ commutes with $u_{1,d}$. 
\end{proof}

We illustrate the utility of Lemma \ref{lem:unipotent} to study representations of $\SL_d(\bfZ)$ by the following two corollaries, which will both play a role in the proof of Theorem \ref{thm:SLdZ}.

\begin{cor}\label{cor:big-to-small}
	For $d\ge3$ and $m<d$, all homomorphisms $\phi \colon \SL_d(\bfZ)\ra \GL_m(\bfC)$ are trivial.
\end{cor}
Under the additional assumption that $\phi$ factors through $\SL_m(\bfZ)\le \GL_m(\bfC)$, this corollary is proved in \cite[Lemma 3]{weinberger}  using superrigidity and the congruence subgroup property.
\begin{proof}[Proof of \cref{cor:big-to-small}]
	If $d \geq 4$ then $\phi(E_{1d})$ is unipotent by \cref{lem:unipotent} \ref{enum:unipotent-i} and since the $E_{ij}$ are conjugate in $\SL_d(\bfZ)$ so are all $\phi(E_{ij})$. We then apply \cref{lem:unipotent} \ref{enum:unipotent-ii} to see that $\phi(E_{1d})$ is trivial, so also the conjugates $\phi(E_{ij})$ are. As the $E_{ij}$ generate $\SL_d(\bfZ)$ the result follows. 
	For $(d,m) = (3,1)$ we use that $\SL_3(\bfZ)$ is perfect (see \cref{lem:low-degree-homology}) and $\GL_1(\bfC)$ is abelian. For $(d,m) = (3,2)$ we apply the first part of \cref{lem:dim3}: in case \ref{enum:dim3firsti} we proceed as for $d=4$ and the case \ref{enum:dim3firstii} is ruled out because the images of $E_{12}$ and $E_{13}$ are not conjugate.
\end{proof}

\begin{cor}\label{cor:apply-lemma}Fix $d\ge3$ and a nontrivial homomorphism $\phi \colon \SL_d(\bfZ)\ra\SL_d(\bfC)$.
\begin{enumerate}
\item\label{enum:apply-lemma-i} If $d\ge4$, then for all $i\neq j$ the matrix $\phi(E_{ij})$ is unipotent and $\phi(E_{ij})-\id$ has rank $1$. Moreover, after possibly precomposing $\phi$ with $(-)^{-t}$, the matrices $\phi(E_{1d}),\ldots,\phi(E_{d-1,d})$ all have the same fixed set.
\item\label{enum:apply-lemma-ii}  If $d=3$, then the same conclusion holds under the additional assumption $\im(\phi)\subset \SL_d(\bfZ)$.
\end{enumerate}
\end{cor}

\begin{proof}We begin with two observations based on the fact that $E_{ij}\in\SL_d(\bfZ)$ is for all $i\neq j$ conjugate to $E_{1d}$. Firstly, to show the first claim of \ref{enum:apply-lemma-i} and \ref{enum:apply-lemma-ii}, it suffices to consider $\phi(E_{1d})$. Secondly, $\phi(E_{1d})$ is nontrivial since otherwise $\phi$ were trivial as $\SL_d(\bfZ)$ is generated by the $E_{ij}$.

In the case $d\ge4$, it suffices to prove that $\phi(E_{1d})$ is not a scalar, for then everything follows from  \cref{lem:unipotent}, using that $E_{1d}$ is conjugate in $\SL_d(\bfZ)$ to $E_{ij}$ for any $i\neq j$. If $\phi(E_{1d})$ were a scalar, then all $\phi(E_{ij})$ are scalars, so $\phi$ would have image in scalar matrices because the $E_{ij}$ generate $\SL_d(\bfZ)$. But since $E_{1d}$ is a commutator and scalar matrices commute, this would imply $\phi(E_{1d})=[\phi(E_{12}),\phi(E_{2d})]=\id$, which is not the case.

Next we consider the case $d=3$. To show the case $m=3$, for which we imposed the additional assumption $\im(\phi)\le \SL_d(\bfZ)$. It suffices by \cref{lem:unipotent} to prove that the nontrivial matrix $\phi(E_{13})$ is unipotent which we prove by contradiction. We consider the restriction of $\phi$ to $\langle E_{12},E_{23}\rangle\cong\bfU_3$ and consult the classification in \cref{lem:dim3}. Since we assumed that $\phi(E_{13})\neq \id$ is not unipotent, we do not need to consider the case \ref{enum:dim3b}. Cases \ref{enum:dim3c} and \ref{enum:dim3d} of \cref{lem:dim3} can be excluded by showing that for these representations the matrices $\phi(E_{12}),\phi(E_{23}),\phi(E_{13})$ are not all conjugate in $\SL_3(\bfZ)$. In almost all cases this can be seen considering their eigenvalues, except in the case
\[\phi(E_{12})=
\left(\begin{array}{ccc}
1&0&0\\
0&-1&0\\
0&0&-1
\end{array}\right)
\>\>\>\>\>\>
\phi(E_{23})=
\left(\begin{array}{ccc}
0&1&0\\
1&0&0\\
0&0&-1
\end{array}\right).\]
Also these matrices are not conjugate in $\SL_3(\bfZ)$ which one can see by reducing modulo $2$.
\end{proof}

\begin{thm}\label{thm:rep}
Fix $d\ge3$ and a nontrivial homomorphism $\phi:\SL_d(\bfZ)\ra\SL_d(\bfZ)$. There exist linearly independent vectors $v_1,\ldots,v_d\in\bfZ^d$ so that, after possibly after precomposing $\phi$ with $(-)^{-t}$, the image of $\phi$ preserves the lattice $\Lambda=\bfZ\{v_1,\ldots, v_d\}$ and for all $A\in\SL_d(\bfZ)$ the matrix of the restriction $\phi(A)|_{\Lambda}$ with respect to the basis $v_1,\ldots,v_d$ is $A$.
\end{thm}

\begin{rem}\label{rmk:counterexample}
One might suspect that given $\phi:\SL_d(\bfZ)\ra\SL_d(\bfC)$ there exists a basis $v_1,\ldots,v_d$ for $\bfC^d$ so that the same conclusion of \cref{thm:rep} holds (this is claimed in the MathOverflow post mentioned in Remark \ref{rmk:bader}). This is not the case. For example, there is a nontrivial representation $\phi:\SL_3(\bfZ)\ra\SL_3(\bfC)$ with finite image, constructed by setting
\[\begin{array}{llll}
\phi(E_{12})=&
\left(\begin{array}{ccc}
1&0&0\\
0&-1&0\\
0&0&-1
\end{array}\right)
&
\phi(E_{23})=&
\left(\begin{array}{ccc}
0&1&0\\
1&0&0\\
0&0&-1
\end{array}\right)
\\
\phi(E_{32})=&
\left(\begin{array}{ccc}
-1&0&0\\
0&0&1\\
0&1&0
\end{array}\right)&
\phi(E_{21})=&
\left(\begin{array}{ccc}
-\frac{1}{2}&-\frac{1}{2}&\frac{-1-\sqrt{-7}}{4}\\
-\frac{1}{2}&-\frac{1}{2}&\frac{1+\sqrt{-7}}{4}\\
\frac{-1+\sqrt{-7}}{4}&\frac{1-\sqrt{-7}}{4}&0
\end{array}\right)
\end{array}\]
and then defining $\phi(E_{13})\coloneqq[\phi(E_{12}),\phi(E_{23})]$ and $\phi(E_{31})\coloneqq[\phi(E_{32}),\phi(E_{21})]$. One can then check directly that this extends to a morphism $\SL_3(\bfZ)\ra\SL_3(\bfQ(\sqrt{-7}))\subset\SL_3(\bfC)$ by checking that these matrices satisfy the relations in the standard presentation of $\SL_3(\bfZ)$ in terms of $E_{ij}$ (see \cite[Corollary 10.3]{MilnorAlgebraic}). This peculiar representation has finite image because for each $i\neq j$, the matrix $\phi(E_{ij})$ has order 2, and the subgroup generated by $\{E_{ij}^2\}$ has finite index in $\SL_3(\bfZ)$ by a general theorem of Tits \cite{tits} (see also \cite[Theorem 3]{meiri}).
\end{rem}

\begin{proof}[Proof of Theorem \ref{thm:rep}] Fix a nontrivial homomorphism $\phi:\SL_d(\bfZ)\ra\SL_d(\bfZ)$. We write $u_{ij}=\phi(E_{ij})$, considered as a matrix in $\SL_d(\bfC)$.  After possibly precomposing $\phi$ with $(-)^{-t}$, we know from \cref{cor:apply-lemma}, that for $i\neq j$, the matrix $u_{ij}$ is unipotent and $u_{ij}-\id$ has rank $1$, and that $H_{1d}=H_{2d}=\cdots=H_{d-1,d}$ where $H_{ij}\le \bfC^d$ be the fixed set of the matrix $u_{ij}$ for $i\neq j$. Note that each $H_{ij}$ is $(d-1)$-dimensional, since $u_{ij}-\id$ has rank 1. Using the fact that for each fixed $1\le k\le d$, the matrices $E_{1k}, E_{2k},\ldots,E_{dk}$ (skipping $E_{kk}$) are simultaneously conjugate to $E_{1d},\ldots,E_{d-1,d}$, we find that also the $d-1$ hyperplanes $H_{1k},H_{2k},\ldots,H_{dk}$ (skipping $H_{kk}$) all agree. We abbreviate this hyperplane by $H_k$. Next we claim that the intersection of hyperplanes
$L_i=H_1\cap\cdots\cap \what H_i\cap\cdots\cap H_d$ for $1\le i\le d$ are all lines. For this it suffices to show that $H_1\cap\cdots\cap H_d$ is trivial. Assume by contradiction that this intersection is nontrivial. By construction, it is the common fixed set for the $u_{ij}$ for all $i\neq j$, so it is in fact fixed by the whole image of $\phi$ since the $u_{ij}=\phi(E_{ij})$ generate the image because the $E_{ij}$ generate $\SL_d(\bfZ)$. Moreover, since the $H_i$ are defined over $\bfQ$, also $L\coloneq H_1\cap\cdots\cap H_d\cap\bfZ^d$ is nontrivial, so the free abelian group $\bfZ^d/L$ has rank $<d$. Combining this with \cref{cor:big-to-small}, we see that the morphism $\SL_d(\bfZ)\ra \SL(\bfZ^d/L)$ induced by $\phi$ is trivial, so $\phi$ factors over the additive group $\Hom(\bfZ^d/L, L)$. The latter is abelian, so $\phi$ must be trivial since $\SL_d(\bfZ)$ is perfect (see \cref{lem:low-degree-homology}). This contradicts our choice of $\phi$.

\smallskip

\noindent {\it Claim.} The image of $u_{ij}-\id$ is $L_i$.

\smallskip

\noindent {\it Proof of Claim.}
For definiteness, we prove the statement for $u_{1d}$. Since $u_{1d}-\id$ has rank $1$ and $L_1=H_2\cap\ldots \cap H_d$ is $1$-dimensional, it suffices to show that the image of $u_{1d}-\id$ is contained in $H_j$ for all $j\neq 1$. Recall that $H_j=H_{1j}$ is the fixed set of $u_{1j}$. Since $u_{1j}$ commutes with $u_{1d}$, the matrix $u_{1j}$ preserves the image of $u_{1d}-\id$, but since this image is only one dimensional, it is an eigenspace for $u_{1j}$, which implies $\im(u_{1d}-\id)\subset H_{1j}$ since $u_{1j}$ is unipotent. This proves the claim.

\smallskip

 Now we construct the basis $v_1,\ldots,v_d$. Fix a nonzero vector $v_d\in L_d$ which we may choose to be an integer vector as $L_d$ is defined over $\bfQ$ since $\phi$ has image in $\SL_d(\bfZ)$. Now define inductively $v_{i}=(u_{i,i+1}-\id)(v_{i+1})\in L_{i}$. Note that the $v_i$ are integer vectors as $u_{i,i+1}=\phi(E_{i,i+1})\in\SL_d(\bfZ)$. Moreover, each $v_i$ is nonzero: if $v_i$ were trivial then $v_{i+1}$ would be contained in $L_{i+1}\cap H_{i+1}=H_1\cap\cdots\cap H_d$ which we saw above is trivial, so we get $v_{i+1}=0$ and inductively $v_d=0$ which is not true. Now we examine what properties the vectors $v_1,\ldots,v_d$ have. First observe that they form a basis for $\bfC^d$, by the general fact that if $H_1,\ldots,H_d$ are hyperplanes of $\bfC^d$ with trivial intersection, then a choice of nonzero vector from each of the lines $L_i = H_1\cap\cdots\cap\what{H_i}\cap\cdots \cap H_d$ gives a basis for $\bfC^d$. By construction, with respect to the basis $v_1,\ldots,v_d$, the matrix of $u_{i,i+1}=\phi(E_{i,i+1})$ is $E_{i,i+1}$. Now using the commutator relations in $\bfU_d$, we conclude that after this change of basis the restriction of $\phi$ to upper triangular matrices is the inclusion, so to finish the proof suffices to show the same for the lower triangular matrices since $\SL_d(\bfZ)$ is generated by upper and lower triangular matrices. As for upper triangular matrices, it suffices to consider $E_{i+1,i}$ for every $i$.  By construction, $\phi(E_{i+1,i})$ has fixed set
$\langle v_1,\ldots,\what{v_i},\ldots,v_d\rangle$ and $(\phi(E_{i+1,i})-\id)(v_i)=a_i\cdot v_{i+1}$ for some scalar $a_i$, so we are left to show $a_i=1$. This follows from the braid relation $E_{i,i+1}^{-1}\>E_{i+1,i}\>E_{i,i+1}^{-1}=E_{i+1,i}\>E_{i,i+1}^{-1}\>E_{i+1,i}$.
\end{proof}

\begin{proof}[Proof of Theorem \ref{thm:SLdZ}]
Fix a nontrivial homomorphism $\phi \colon \SL_d(\bfZ)\ra\SL_d(\bfZ)$ and let $v_1,\ldots,v_d\in\bfZ^d$ be the linearly independent vectors promised  by \cref{thm:rep}, so that possibly after precomposing $\phi$ with $(-)^{-t}$, the matrix $\phi(A)$ for $A\in\SL_d(\bfZ)$ preserves the lattice $\Lambda=\bfZ\{v_1,\ldots,v_d\}$, and the restriction $\phi(A)|_{\Lambda}$ is represented by the matrix $A$ when written in the basis $v_1,\ldots,v_d$. In particular, this has as consequence that every orientation-preserving automorphism of $\Lambda\le \bfZ^d$ extends to an orientation-preserving automorphism of $\bfZ^d$. We claim that this in turn implies $\Lambda=\ell\bfZ^d$ for some $\ell>0$. Dividing the basis by $\ell$, this would show that we can choose  $v_1,\ldots,v_d$ to form a basis of $\bfZ^d$, so $\phi$ is given by conjugation by an element of $\GL_d(\bfZ)$. That  $\Lambda=\ell\cdot \bfZ^d$ for some $\ell>0$ follows from two facts: (a) for every non-characteristic subgroup $L\subset\bfZ^d$ of full rank, there exists an (orientation-preserving) automorphism of $L$ that does not extend to $\bfZ^d$, so $\Lambda$ has to be characteristic, and (b) every characteristic subgroup $L\le\bfZ^d$ of full rank has the form $\ell\cdot \bfZ^d$ for some $\ell>0$. To see these two facts, we fix a subgroup $L\le \bfZ^d$ of full rank. By the elementary divisor theorem, there is a basis $b_1,\ldots, b_d$ of $\bfZ^d$ and natural numbers $\ell_1,\ldots,\ell_d$ such that $\ell_1\cdot b_1,\ldots ,\ell_d\cdot b_d$ is a basis of $L$. If $L$ is non-characteristic, then $\ell_i\neq \ell_j$ for some $i$ and $j$ (since $\ell\cdot \bfZ^d\le \bfZ^d$ is clearly characteristic), so the automorphism of $L$ that interchanges $\ell_i\cdot b_i$ and $\ell_j\cdot b_j$ does not extend to $\bfZ^d$ (by interchanging a second pair of basis vectors we also find an orientation-preserving example of such an automorphism). This shows (a). Moreover, if we assume $\ell_i\neq \ell_j$ for some $i$ and $j$, then the automorphism of $\bfZ^d$ that interchanges $b_i$ and $b_j$ does not restrict to $L$, so $L$ cannot be characteristic. This shows (b).
\end{proof}

\bibliographystyle{amsalpha}
\bibliography{literature}
\vspace{0.1cm}
\end{document}